\documentclass[10pt]{amsart}
\usepackage{pifont}
\usepackage{amssymb}
\usepackage[mathscr]{euscript}
\usepackage[all]{xy}
\usepackage[dvips]{graphics}
\usepackage{amsmath,amssymb}
\usepackage[dvips,final]{graphicx}
\usepackage[dvips]{geometry}
\usepackage{color}
\usepackage{epsfig}
\usepackage{latexsym}
\usepackage{pstricks}
\usepackage{multirow}
\usepackage{rotating}
\usepackage{hyperref}
\usepackage{enumerate}

\newtheorem{theorem}{Theorem}
\newtheorem{lemma}[theorem]{Lemma}
\newtheorem{proposition}[theorem]{Proposition}
\newtheorem{corollary}[theorem]{Corollary}

\theoremstyle{definition}
\newtheorem{definition}{Definition}[section]

\theoremstyle{remark}
\newtheorem{remark}{Remark}[section]

\title[Monoid of Long Virtual Knots]{Prime Decomposition and Non-Commutativity\\ in the Monoid of Long Virtual Knots}
\author{Micah W. Chrisman}
\begin{document}
\begin{abstract} It is well-known that the monoid of long virtual knots is not commutative. This contrasts with the case of classical long knots, where $A \# B \leftrightharpoons B \# A$ for all $A,B$. In the present paper, we present a new proof that two inequivalent non-classical prime long virtual knots never commute. The original result is due to Manturov. The techniques used here are mostly geometric. First, a slightly strengthened version of Kuperberg's theorem is established. We then show that a well-defined concatenation of two long knots in a thickened surface is preserved by stabilization when both long knots are non-classical. Finally, it is proved that if $A,B,C,D$ are prime non-classical long virtual knots such that $A \# B$ is non-classical and $A \# B \leftrightharpoons C \# D$, then $A \leftrightharpoons C$ and $B \leftrightharpoons D$.
\end{abstract}
\keywords{long virtual knots, Kuperberg's Theorem, prime decomposition}
\subjclass[2000]{57M25,57M27}
\maketitle
\section{Introduction}
\subsection{Motivation} A well-known fact from classical knot theory is that if $A$ and $B$ are long knots, then $A \# B \leftrightharpoons B \# A$. Here the operation $\#$ denotes concatenation and $\leftrightharpoons$ denotes Reidemeister equivalence. In virtual knot theory \cite{KaV}, the operation is not commutative. The first counterexample is due Manturov \cite{long}. Since then, other authors have been able to detect this phenomenon with other invariants \cite{quatlong,allsamall}. When $A$ and $B$ are distinct irreducibly odd free long knots, non-commutativity of $A$ and $B$ follows from a purely combinatorial argument (for definitions, see \cite{Ma2}).
\newline
\newline
If $A$ is a long classical knot and $B$ is any long virtual knot, then it is true that $A \# B \leftrightharpoons B \# A$ \cite{manturov_compact_long}.  This follows from the fact that $A \# B$ may be lifted to a knot in a thickened surface $\Sigma \times I$.  The classical knot can then be ``pushed'' along the knot until it goes from one side of $B$ to the other. As this is $B \# A$, we see that every long classical knot commutes with every long virtual knot. This is depicted at the level of long virtual knot diagrams in Figure \ref{class_virt_comm}.

\begin{figure}[htb]
\[
\begin{array}{ccccc}
\begin{array}{c}\scalebox{.6}{\psfig{figure=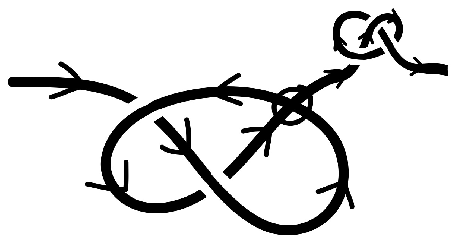}}\end{array} & \leftrightharpoons & \begin{array}{c}\scalebox{.6}{\psfig{figure=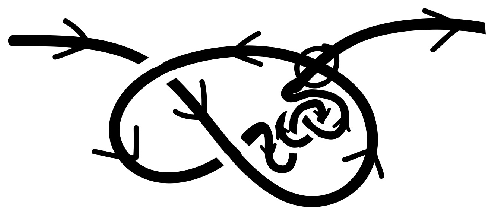}} \end{array} & \leftrightharpoons & \begin{array}{c}\scalebox{.6}{\psfig{figure=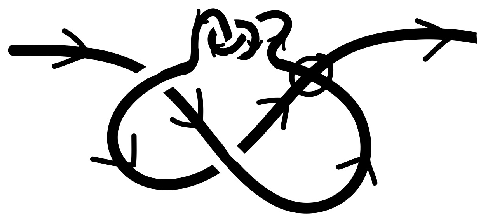}} \end{array} \\
\begin{array}{c}\scalebox{.6}{\psfig{figure=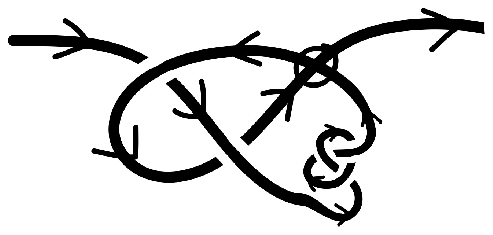}} \end{array} & \leftrightharpoons & \begin{array}{c}\scalebox{.6}{\psfig{figure=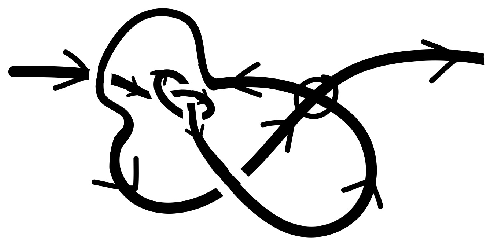}} \end{array} & \leftrightharpoons  & \begin{array}{c}\scalebox{.6}{\psfig{figure=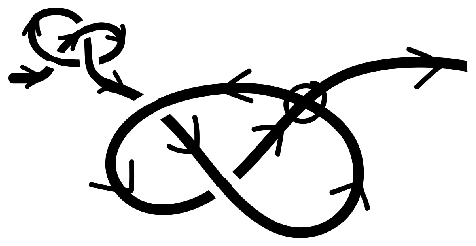}} \end{array} \\ 
\end{array}
\]
\caption{Long classical knots commute with all long virtual knots.} \label{class_virt_comm}
\end{figure}
\hspace{1cm}
\newline
Therefore, there are some cases for which the product of long virtual knots is commutative and some for which it is not.  We will investigate a particular case in which long virtual knots never commute.

\subsection{Preliminary Definitions} The investigation will be carried out with techniques of geometric topology in the piecewise linear category. As seen in the example that long virtual knots and long classical knots always commute, we will look at long virtual knots both as diagrams in $\mathbb{R}^2$ and as long knots in thickened surfaces. Before stating the main results, we present the corresponding definitions of each interpretation of long virtual knots.

\subsubsection{Long Virtual Knots: Interpretation as Diagrams in $\mathbb{R}^2$} 

\begin{definition}[Long Virtual Knot Diagrams, Long Virtual Knots] A \emph{long virtual knot diagram} is an immersion $K:\mathbb{R} \to \mathbb{R}^2$ whose image coincides with the $x$-axis outside some closed ball $B(0,R)$ and such that each double point is represented as a classical crossing or virtual crossing. All long virtual knot diagrams are assumed to be oriented from $-\infty$ to $\infty$. A \emph{long virtual knot} is an equivalence class of long virtual knot diagrams modulo the Reidemeister moves and the detour move \cite{KaV}. If $K_1,K_2$ are in the same equivalence class for this relation, we write $K_1 \leftrightharpoons K_2$. The set of equivalence classes of long virtual knots is denoted by $\mathscr{K}$.
\end{definition}

\begin{definition}[Long Classical Knot, Long Trivial Knot] A long virtual knot diagram having no virtual crossings is said to be a \emph{long classical knot diagram}. If a long virtual knot diagram is in the same equivalence class as a long classical knot diagram, then it is said to be \emph{classical}. If there is no long classical knot diagram to which $K$ is equivalent, the $K$ is $\emph{non-classical}$. The \emph{long trivial knot} is the long knot $\mathbb{R} \to \mathbb{R}^2$ defined by $t \to (t,0)$. We will often denote the trivial long knot by $\to$. 
\end{definition}

\begin{definition}[Concatenation of Diagrams] \label{defn_concat_virtual} Let $\mathbb{H}_L:=\{(x,y)\in \mathbb{R}^2:x<0\}$ and $\mathbb{H}_R:=\{(x,y) \in \mathbb{R}^2:x>0\}$. Let $f_L:\mathbb{R}^2 \to \mathbb{H}_L$ be the function $f_L(x,y)=(-e^{-x},y)$ and $f_R:\mathbb{R}^2 \to \mathbb{H}_R$ be the function $f_R(x,y)=(e^x,y)$. Let $K_1$, $K_2$ be long virtual knot diagrams.  The \emph{concatenation} of $K_1$ and $K_2$ is the long virtual knot diagram $K_1 \# K_2$ defined by: 
\[
K_1 \# K_2(t)=\left\{\begin{array}{cc} f_L \circ K_1(-\ln(-t)), & \text{if } t<0 \\ 0, & \text{if } t=0 \\ f_R \circ K_2(\ln(t)), & \text{if } t>0 \end{array} \right.
\]
where the double points of $K_1 \# K_2$ in $\mathbb{H}_L$ are given the classical crossing and virtual crossing structure of the corresponding double points of $K_1$ and the double points of $K_1\#K_2$ in $\mathbb{H}_R$ are given the classical crossing and virtual crossing structure of the corresponding double points of $K_2$.
\end{definition}

\begin{definition}[Linearly Prime] A long virtual knot $K$ is said to be \emph{linearly prime} (see also \cite{KaV}) if $K$ is not the long trivial knot and whenever $K \leftrightharpoons K_1 \# K_2$, then either $K_1 \leftrightharpoons \to$ or $K_2 \leftrightharpoons \to$.
\end{definition} 

\subsubsection{Long Virtual Knots: Thickened Surface Interpretation} Let $I=[0,1]$. Let $\Sigma$ be a connected compact oriented surface with either one or two distinguished boundary components. Let $\mathscr{S}$ denote the set of such surfaces. Let $c(\Sigma)$ be the number of boundary components of $\Sigma$ and $num(\Sigma)$ the number of distinguished boundary components. The \emph{genus} of $\Sigma$, denoted $g(\Sigma)$ is defined by the equation $\chi(\Sigma)=2-2g(\Sigma)-c(\Sigma)$, where $\chi(\Sigma)$ denotes the Euler characteristic. 

\begin{definition}[Long Knot in $\Sigma \times I$] Let $\Sigma \in \mathscr{S}$ and let $\mathscr{C}(\Sigma)$ be the set of annuli $C \times I \subset \Sigma \times I$ where $C$ is a distinguished boundary component of $\Sigma$. A \emph{long knot} in $\Sigma \times I$ is a (p.l.) embedding $\tau:I \to \Sigma \times I$ satisfying the following properties:
\begin{enumerate}
\item $\text{im}(\tau) \cap \partial (\Sigma \times I)=\{\tau(0),\tau(1)\}$,
\item $\tau(0),\tau(1) \in \text{int}(C \times I)$, $\tau(0) \ne \tau(1)$, if $\emph{num}(\Sigma)=1$ and $\mathscr{C}(\Sigma)=\{C \times I\}$,  
\item $\tau(0) \in \text{int}(C_1 \times I)$, $\tau(1) \in \text{int}(C_2 \times I)$, if $\emph{num}(\Sigma)=2$ and $\mathscr{C}(\Sigma)=\{C_1 \times I,C_2 \times I\}$. 
\end{enumerate}
A long knot $\tau$ in $\Sigma \times I$ will be denoted by the pair $(\Sigma,\tau)$. 
\end{definition}

\begin{definition}[Equivalence of Long Knots in $\Sigma \times I$] Long knots $\tau_1,\tau_2:I \to \Sigma \times I$ are said to be \emph{equivalent} if:
\begin{enumerate} 
\item there is an orientation preserving homeomorphism $h:\Sigma \to \Sigma$ mapping each boundary component to itself and satisfying $(h \times \text{id}) \circ \tau_1=\tau_2$, or
\item there is a (p.l.) ambient isotopy $H:(\Sigma \times I)\times I \to \Sigma \times I$ such that for all $t$ and all components $C$ of $\partial \Sigma$, $H_t(C\times I)=C \times I$ and $H_t|_{\Sigma \times \partial I}=\text{id}_{\Sigma \times \partial I}$, or
\item there is a finite sequence of equivalences as in (1) and (2).
\end{enumerate}
\end{definition}

\begin{definition}[Vertically Proper] A (p.l) embedding of a disc $F:I \times I \to \Sigma \times I$ is said to be \emph{vertically proper} if $F(I \times \{0\}) \subset  \Sigma \times \{0\}$, $F(I \times \{1\}) \subset \Sigma \times \{1\}$, $F(\{0\} \times  I) \subset (\partial \Sigma) \times I$, and $F(\{1\} \times I) \subset (\partial \Sigma) \times I$. A (p.l) embedding of an annulus $F:S^1 \times I \to \Sigma \times I$ is said to be \emph{vertically proper} if $F(S^1 \times \{0\}) \subset \Sigma \times \{0\}$ and $F(S^1 \times \{1\}) \subset \Sigma \times \{1\}$. 
\end{definition}

\begin{definition}[Destabilization, Stabilization] A \emph{ destabilization} \cite{kuperberg} of $(\Sigma,\tau)$ is cutting $\Sigma \times I$ along a vertically proper disc or annulus $W$ such that $\text{im}(\tau) \cap W=\emptyset$ and discarding any resulting connected components which do not contain $\text{im}(\tau)$. The inverse operation of a destabilization is a \emph{stabilization}. After destabilization, we obtain a new long knot $(\Sigma',\tau')$ such that $\Sigma'$ is homeomorphic to the surface obtained by cutting $\Sigma \times \{1\}$ along $W \cap (\Sigma \times \{1\})$. We will say that a destabilization is \emph{inessential} if it cuts off a $3$-ball from $\Sigma \times I$. The inverse operation is called an \emph{inessential stabilization}. 
\end{definition}

\begin{remark} After an annular destabilization, it is customary to cap the resulting manifold by gluing thickened discs $D^2 \times I$ along $(\partial D^2) \times I$ \cite{kuperberg}. Here we instead consider capping as a distinct operation. Indeed, it is an inessential stabilization.
\end{remark}

\begin{remark} An annular destabilization does not change the number of distinguished boundary components. However, a destabilization on a disc can increase, decrease, or preserve the number of distinguished boundary components.
\end{remark}

\begin{definition}[Descendant] If $(\Sigma',\tau')$ can be obtained from $(\Sigma,\tau)$ by a sequence of destablizations, inessesential stabilizations, o.p. homeomorphisms of surfaces, and equivalences of long knots in thickened surfaces, then we say that $(\Sigma',\tau')$ is a \emph{descendant} of $(\Sigma,\tau)$ (i.e. there are no essential stabilizations allowed).
\end{definition}

\begin{definition}[Genus Reducing, Num reducing] A destabilization of $(\Sigma,\tau)$ to $(\Sigma',\tau')$ is said to be \emph{genus reducing} if $g(\Sigma')<g(\Sigma)$. The destabilization $(\Sigma,\tau)$ to $(\Sigma',\tau')$ is said to be \emph{num reducing} if $num(\Sigma')<num(\Sigma)$. 
\end{definition}

\begin{definition}[Irreducible Long Knot in $\Sigma \times I$] We say that $(\Sigma,\tau)$ is \emph{irreducible} if every destabilization of $(\Sigma,\tau)$ is not genus reducing and not num reducing.
\end{definition}

\begin{definition}[Stable Equivalence, $\mathscr{K}(\mathscr{S})$] Two long knots in thickened surfaces $(\Sigma_1,\tau_1)$, $(\Sigma_2,\tau_2)$ are said to be \emph{stably equivalent} if they may be obtained from one another by a finite sequence of equivalences of long knots in thickened surfaces, stabilizations, destabilizations, and maps of the form $h \times \text{id}:\Sigma \times I \to \Sigma' \times I$ where $h$ is an orientation preserving homeomorphism $h:\Sigma\to \Sigma'$ for $\Sigma,\Sigma' \in \mathscr{S}$. If $(\Sigma_1,\tau_1)$, $(\Sigma_2,\tau_2)$ are stably equivalent, we write $(\Sigma_1,\tau_1)\sim_s(\Sigma_2,\tau_2)$. The set of stable equivalence classes of long knots is denoted by $\mathscr{K}(\mathscr{S})$.
\end{definition}

\begin{definition}[Concatentation] Let $(\Sigma,\tau)$, $(\Sigma_1,\tau_1)$, $(\Sigma_2,\tau_2)$ be long knots in thickened surfaces. We write $(\Sigma,\tau)=(\Sigma_1\#\Sigma_2,\tau_1\#\tau_2)$ if there is a vertically proper disc $R$ in $\Sigma \times I$ such that the following conditions are satisfied:
\begin{enumerate}
\item there are orientation preserving embeddings $i_1: \Sigma_1 \times I \to \Sigma \times I$ and $i_2:\Sigma_2 \times I \to \Sigma \times I$ such that $\text{im}(i_1)\cup\text{im}(i_2)=\Sigma \times I$, $\text{im}(i_1)\cap\text{im}(i_2)=R$, $i_j(\Sigma_j \times \{s\}) \subset \Sigma \times \{s\}$ for $j=1,2$, $s=0,1$,
\item there is exactly one point $t_R \in (0,1) \subset I$ such that $\tau(t_R) \in R$, and
\item $\tau(t)=i_1(\tau_1\left(\frac{t}{t_R}\right))$ for $0 \le t \le t_R$ and $\tau(t)=i_2(\tau_2\left(\frac{t-t_R}{1-t_R}\right))$ for $t_R \le t \le 1$.  
\end{enumerate}
In this case, we say that $(\Sigma,\tau)$ is a \emph{concatenation} of $(\Sigma_1,\tau_1)$ and $(\Sigma_2,\tau_2)$. The vertically proper surface $R$ is called the \emph{surface defining the concatenation}. The long knot $(\Sigma_1,\tau_1)$ is called the \emph{left part} of the concatenation and the long knot $(\Sigma_2,\tau_2)$ is called the \emph{right part} of the concatenation. A concatenation is depicted in Figure \ref{cat_fig}.
\end{definition}

\begin{figure}[h]
\[
\begin{array}{ccc}
\begin{array}{c}\psfig{figure=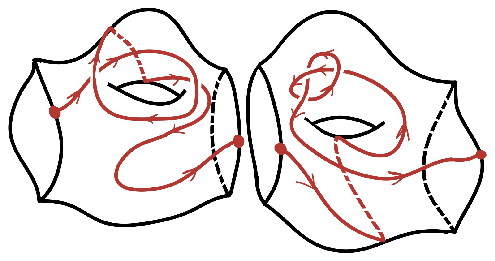}  \end{array} & \to & \begin{array}{c}\psfig{figure=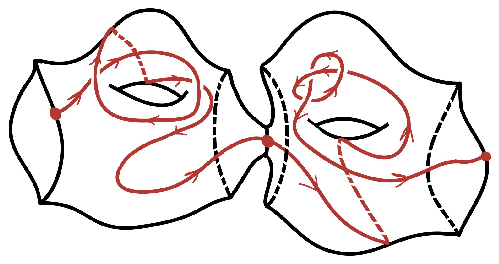}  \end{array}\\
\end{array}
\]
\caption{A concatenation $(\Sigma,\tau)=(\Sigma_1\#\Sigma_2,\tau_1 \# \tau_2)$} \label{cat_fig}
\end{figure}

\subsection{Statement of Main Results} The following theorem was proved by Manturov in unpublished work (see also Problem 20 of \cite{FKM}). Our goal is to present a new proof of this result.

\begin{theorem} \label{main_3} Let $A,B,C,D$ be non-classical linearly prime long virtual knots such that $A \# B$ is non-classical. If $A \# B \leftrightharpoons C \# D$, then $A \leftrightharpoons C$ and $B \leftrightharpoons D$. In particular, distinct non-classical linear primes do not commute in the monoid of long virtual knots. 
\end{theorem}

\noindent As a convenience to the reader, we endeavor to provide a self-contained proof of the result. We wish to minimize as much as possible the number of appeals to the literature while citing all the relevant work which is known to the author. The proof is divided into several main lemmas. First, it is shown how to move from the diagrammatic interpretation of long virtual knots to the thickened surface interpretation. In particular, it must be shown how concatenations affect long knots in thickened surfaces which correspond to linearly prime long virtual knots.  The correspondence is constructed in a similar way to the correspondence between virtual knots and stability classes of knots in thickened surfaces (see \cite{kamkam},\cite{CKS}).

\begin{lemma}\label{prime_corr} There is a one-to-one correspondence $\mathscr{F}:\mathscr{K} \to \mathscr{K}(\mathscr{S})$ such that if $\mathscr{F}^{-1}(\Sigma,\tau)$ is a linearly prime long virtual knot and $(\Sigma,\tau)=(\Sigma_1\#\Sigma_2,\tau_1\#\tau_2)$ for some long knots $(\Sigma_1,\tau_1),(\Sigma_2,\tau_2)$, then either $\mathscr{F}^{-1}(\Sigma_1,\tau_1) \leftrightharpoons \to$ or $\mathscr{F}^{-1}(\Sigma_2,\tau_2) \leftrightharpoons \to$.
\end{lemma}

\noindent Secondly, we prove a slightly stronger version of Kuperberg's theorem \cite{kuperberg} for long virtual knots. The slight strengthening is useful for our purposes. It shows that we can find a ``unique'' irreducible descendant of any long knot in a thickened surface in the sense that the surface is of minimal genus and among those surfaces of minimal genus we have the smallest possible number of distinguished boundary components.

\begin{lemma}\label{main_1} Let $(\Sigma,\tau)$ be a long knot in $\Sigma \times I$. Then all of the following hold:
\begin{enumerate}
\item Either $(\Sigma,\tau)$ is irreducible or it has an irreducible descendent.
\item If $(\Sigma',\tau')\sim_s (\Sigma,\tau)$, then they have a common irreducible descendent.
\item If $(\Sigma,\tau)$ has more than one irreducible descendent, then they may be obtained from one another by inessential stabilizations/destabilizations, o.p. homeomorphisms of surfaces, and equivalences of long knots in thickened surfaces.
\item The genus and number of distinguished boundary components are the same for all irreducible $(\Sigma_0,\tau_0)$ which are stably equivalent to $(\Sigma,\tau)$.   
\end{enumerate}
\end{lemma}

\noindent Lastly, we prove a slightly stronger version of a theorem due to Manturov on destablizations of decompositions of virtual knots \cite{kauffmant,manturov_compact_long}. Manturov has shown that for the two types of decompositions of virtual knots, a destabilization either preserves the decomposition and destabilizes one of the components, or it converts one type of decomposition to the other. For long knots in thickened surfaces, we show that destabilizations preserve decompositions when both components are non-classical. Moreover, an irreducible descendent must inherit the decomposition.   

\begin{lemma} \label{main_2} Let $(\Sigma_1,\tau_1)$, $(\Sigma_2,\tau_2)$, and $(\Sigma,\tau)=(\Sigma_1\#\Sigma_2,\tau_1\#\tau_2)$ be long knots in thickened surfaces which stabilize to non-classical long virtual knots. Then there is an irreducible $(\Sigma_0,\tau_0)$ such that:
\begin{enumerate}
\item $(\Sigma_0,\tau_0)$ is a descendant of $(\Sigma_1 \# \Sigma_2,\tau_1 \# \tau_2)$, 
\item $(\Sigma_0,\tau_0)=(\Sigma_1' \# \Sigma_2',\tau_1' \# \tau_2')$, and
\item $(\Sigma_i',\tau_i')\sim_s (\Sigma_i,\tau_i)$ for $i=1,2$. 
\end{enumerate}
\end{lemma}

\noindent The proof of Theorem \ref{main_3} is then reduced to considering the decompositions of irreducible representatives.  This is established using the techniques of geometric topology in the piecewise linear category.
\newline
\newline
There has been much recent and important work on the existence and uniqueness of prime decompositions of virtual knots \cite{kauffmant,koramat,manturov_compact_long,roots,matveevprime}.  The present paper uses similar geometric techniques but applies them to long virtual knots.
\newline
\newline
The outline of the paper is as follows. Section \ref{sec_long_knots} contains the proof of Lemma \ref{prime_corr}. In section \ref{sec_geom_construct}, we establish some geometric lemmas needed in the proofs of Lemma \ref{main_1}, Lemma \ref{main_2}, and Theorem \ref{main_3}. Section \ref{sec_kup} contains the proof of Lemma \ref{main_1}. Section \ref{sec_stab_decomp} contains the proof of Lemma \ref{main_2}. Finally, our proof of Theorem \ref{main_3} is disclosed in Section \ref{sec_non_commute}.

\subsection{Acknowledgements} The author is deeply indebted to V. O. Manturov for providing a description of his unpublished proof of Theorem \ref{main_3}. In the author's attempt to reconstruct and generalize Manturov's argument, a somewhat different proof was discovered. This alternate version is presented here. The author suspects that his alternative has no characteristics to be preferred over Manturov's original version. However, it certainly provides a different way of looking at the problem. The author is also indebted to Manturov for carefully reading numerous previous drafts and providing excellent counsel in regards to improving the exposition.
\newline
\newline
The author is also grateful for helpful conversations and correspondences with H. A. Dye, R. Fenn, A. Kaestner, L. Kauffman, S. Nelson, K. Orr, and R. Todd.  

\section{Proof of Lemma \ref{prime_corr}: Correspondence of Decompositions} \label{sec_long_knots}

\noindent The aim of this section is to prove Lemma \ref{prime_corr}, that there is a one-to-one correspondence between long virtual knots and long knots in thickened surfaces modulo stabilization. Furthermore, it is proved that the one-to-one correspondence preserves ``primality''. In particular, we show that if $(\Sigma,\tau)$ is a long knot which corresponds to a linearly prime long virtual knot and $(\Sigma,\tau)=(\Sigma_1\#\Sigma_2,\tau_1\#\tau_2)$, then either $(\Sigma_1,\tau_1)$ stabilizes to the long trivial knot or $(\Sigma_2,\tau_2)$ stabilizes to the long trivial knot.  
\newline
\newline
The one-to-one correspondence factors through stability classes of long knot diagrams on surfaces \cite{kamkam,CKS}. In Section \ref{diag_on_surf}, we give a precise definition of these equivalence classes following \cite{CKS}. We construct in Section \ref{oto_corres} a one-to-one correspondence between long virtual knots and long knot diagrams on surfaces modulo stabilization. Then we construct a one-to-one correspondence between long knot diagrams on surfaces and long knots in thickened surfaces modulo stabilization. Together this makes a one-to-one correspondence between long virtual knots and long knots in thickened surfaces modulo stabilization. The arguments in Section \ref{diag_on_surf} and \ref{oto_corres} are standard \cite{kamkam} in virtual knot theory and can be skipped by experts. Lemma \ref{prime_corr} is proved in Section \ref{sec_decomp_corr}.

\subsection{Long Knot Diagrams on Surfaces} \label{diag_on_surf} In this section, we give the precise definitions of long knot diagrams on surfaces and stable equivalence classes of long knot diagrams on surfaces. 
\newline
\newline
Let $\Sigma$ be a closed connected oriented surface with one or two distinguished boundary components. A \emph{long knot diagram} on $\Sigma$ is an immersion $\tau:I \to \Sigma$ such that each double point is marked as a classical crossing, $\tau(0),\tau(1) \in \partial \Sigma$, $\tau(0)$ and $\tau(1)$ are in the distinguished boundary component if $num(\Sigma)=1$, and $\tau(0),\tau(1)$ are in different distinguished boundary components if $num(\Sigma)=2$.  
\newline
\newline
Two long knot diagrams on a surface $\Sigma$ as above are considered \emph{equivalent} if they may be obtained from one another by a finite sequence of Reidemeister moves on $\Sigma$. Two long knot diagrams $(\Sigma_1,\tau_1)$, $(\Sigma_2,\tau_2)$ are said to be \emph{elementary equivalent} \cite{CKS,kamkam} if there is a connected oriented compact surface $\Gamma$ and orientation preserving embeddings $g_1:\Sigma_1 \to \Gamma$, $g_2:\Sigma_2 \to \Gamma$ such that $g_1(\tau_1)$ and $g_2(\tau_2)$ are Reidemeister equivalent on $\Gamma$. We write $(\Sigma_1,\tau_1)\sim_e(\Sigma_2,\tau_2)$ if they are elementary equivalent. Two long knots $(\Sigma_1,\tau_1)$, $(\Sigma_2,\tau_2)$ are said to be \emph{stably equivalent} if there is a finite sequence:
\[
(\Sigma_1,\tau_1)=(\Gamma_0,\gamma_0)\sim_e(\Gamma_1,\gamma_1)\sim_e\cdots\sim_e(\Gamma_n,\gamma_n)=(\Sigma_2,\tau_2),
\] 
of elementary equivalences between them. If $(\Sigma_1,\tau_1)$ and $(\Sigma_2,\tau_2)$ are stably equivalent, we write $(\Sigma_1,\tau_1)\sim_s(\Sigma_2,\tau_2)$.

\subsection{Correspondence of Long Virtual Knots and Long Knots on Surfaces} \label{oto_corres} To each long virtual knot, we associate a long knot diagram on a surface. The construction is similar to that of the virtual knot case \cite{CKS,kamkam}. It then follows that there is a one-to-one correspondence between long virtual knots and stability classes of long knots in thickened surfaces.
\newline
\newline
A \emph{band-pass presentation} of a long virtual knot diagram $K$ is constructed as follows. Consider $K$ as a long virtual knot diagram on $S^2\backslash\{\infty\}$. Let $U$ be a small coordinate neighborhood of $\infty$ which is chosen so that $U \cap \text{im}(K)\subset \mathbb{R}$ and $\mathbb{R}\backslash (U\cap \text{im}(K))$ is a closed interval. Let $V$ be a regular neighborhood of $\partial U$ in $\mathbb{R}^2 \backslash U$. We define $\partial U$ to be the distinguished boundary component. At each classical crossing of $K$, a ``cross'' is drawn. At each virtual crossing, a pair of overpassing bands is drawn in $\mathbb{R}^3$.  The crosses, bands, and the regular neighborhood $V$ are connected by bands in $\mathbb{R}^2$ along the arcs of $K$ (see Figure \ref{longskimm}). This gives a connected oriented surface $\Sigma_K$ with one distinguished boundary component $\partial U$. Let $\tau_K$ denote the long knot diagram on $\Sigma_K$. The long knot diagram will be denoted by the pair $(\Sigma_K,\tau_K)$. Define $\mathscr{F}_0(K)=(\Sigma_K,\tau_K)$.

\begin{figure}[h]
\[
\begin{array}{|ccc|} \hline
 & & \\
\begin{array}{c}\scalebox{.5}{\psfig{figure=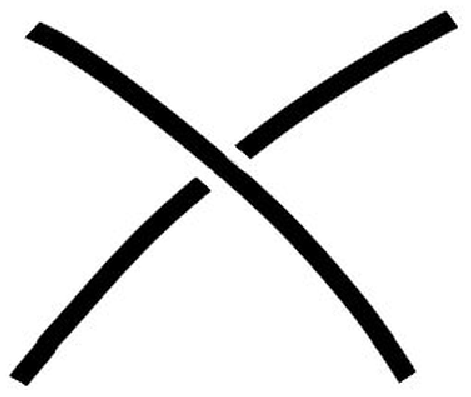}} \end{array} & \to\to & \begin{array}{c}\scalebox{.5}{\psfig{figure=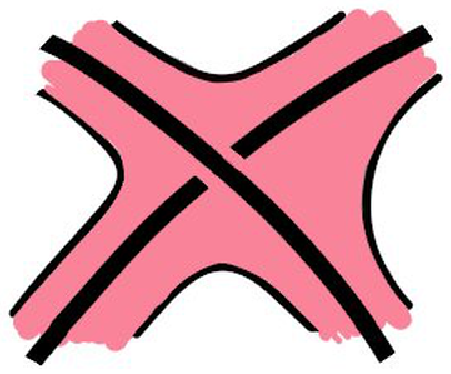}} \end{array}\\
\begin{array}{c}\scalebox{.5}{\psfig{figure=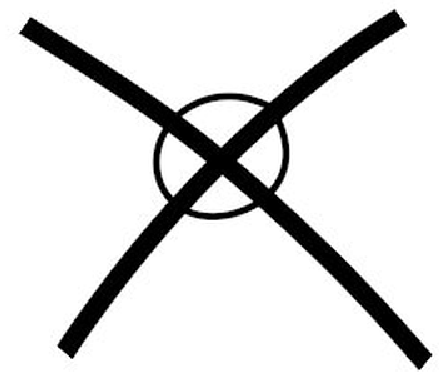}} \end{array} & \to\to & \begin{array}{c}\scalebox{.5}{\psfig{figure=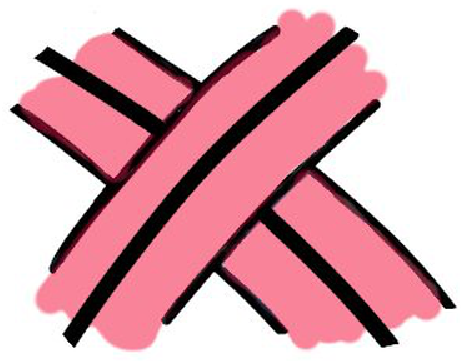}} \end{array}\\ \hline
\end{array}
\]
\[
\begin{array}{|ccccc|} \hline
& & & & \\
\begin{array}{c}\scalebox{.05}{\psfig{figure=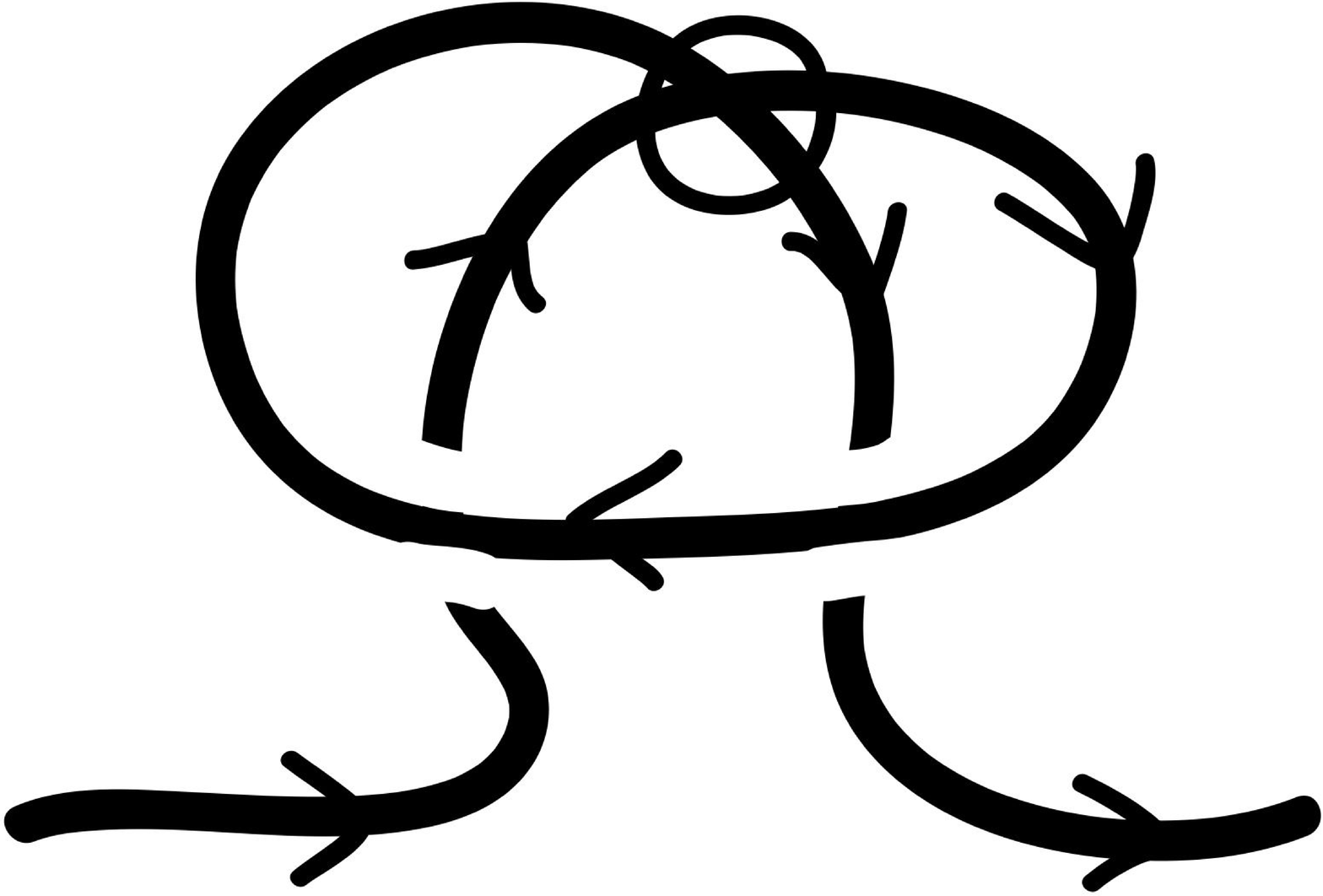}} \end{array} & \to & \begin{array}{c}\scalebox{.75}{\psfig{figure=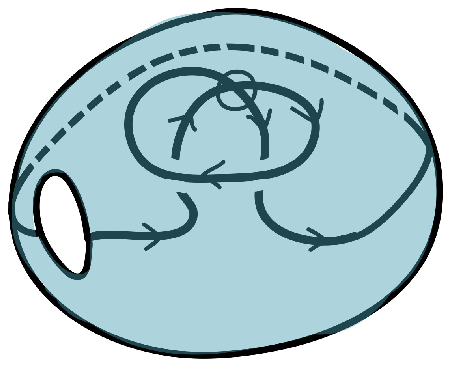}} \end{array} & \to &
\begin{array}{c}\scalebox{.75}{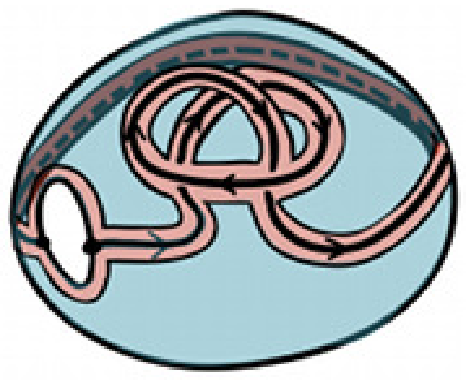} \end{array} \\ \hline
\end{array}
\]
\caption{Construction of a band-pass presentation.} \label{longskimm}
\end{figure}
\hspace{1cm}
\newline
The function $\mathscr{F}_0$ from long virtual knot diagrams to long knot diagrams on surfaces has an inverse (see also \cite{kamkam}). Let $(\Sigma,\tau)$ be a long knot diagram on a surface. There is a compact connected oriented surface $\Gamma$ and an embedding $\Gamma \to \Sigma$ such that $\text{im}(\tau)$ together with the distinguished boundary component(s) of $\Sigma$ is a deformation retract of $\Gamma$. Note that $(\Gamma,\tau)\sim_e(\Sigma,\tau)$.  
\newline
\newline
The next step in the construction of $\mathscr{F}_0^{-1}(\Sigma,\tau)$ is to take a band decomposition of $\Gamma$. Indeed, we have discs which contain individual crossings of $\tau$, bands which contain arcs of $\tau$, and bands along boundary components containing the ends of $\tau$. The band decomposition can be embedded in $\mathbb{R}^3=\mathbb{R}^2 \times \mathbb{R}$ so that the discs around the crossings are contained in $\mathbb{R}^2 \times \{0\}$. 
\newline
\newline
Continuing in the construction of $\mathscr{F}_0^{-1}(\Sigma,\tau)$, if $x$ is a disc containing a crossing in $\Gamma$, let $r(x)$ denote the image of the crossing in $\mathbb{R}^2 \times \{0\}$. Then $\partial x \cap \text{im}(\tau)$ contains four points. If the discs are numbered $x_1,x_2,\ldots, x_m$, then we label the four points in any order as $x_{i,1}$, $x_{i,2}$, $x_{i,3}$, and $x_{i,4}$. We connect $r(x_{i,j})$ to $r(x_{k,l})$ in $\mathbb{R}^2\backslash \cup_{i} r(x_i)$ by a simple (p.l.) arc $a_{i,j,k,l}$ if $x_{i,j}$ and $x_{k,l}$ lie on the same band of $\Gamma$. By a general position argument, the set of all $a_{i,j,k,l}$ may be chosen so that they intersect transversally and any point of intersection is the intersection of exactly two arcs. Now, there is a unique point $x_{i_0,j_0}$ which is connected by a simple (p.l.) arc in $\Gamma$ to $\tau(0)$ and a unique point $x_{i_1,j_1}$ which is connected by a simple (p.l.) arc to $\tau(1)$. Let $B$ be an open ball centered at the origin in $\mathbb{R}^2$ containing all $r(x_i)$ and $a_{i,j,k,l}$. Connect $x_{i_0,j_0}$ to the leftmost point $p_0$ on the $x$-axis of $\partial B \cap (\mathbb{R} \times 0)$ by a simple (p.l) arc $a_0$ in $B \backslash \cup_i r(x_i)$ having only transversal intersections with the $a_{i,j,k,l}$.  Similarly, we connect $x_{i_1,j_1}$ to the rightmost point $p_1$ on the $x$-axis of $\partial B \cap (\mathbb{R} \times \{0\})$ by a simple (p.l.) arc $a_1$ in $B \backslash \cup_i r(x_i)$. All transversal intersections of the $a_{i,j,k,l}$, $a_0$, and $a_1$ are marked as virtual crossings. Let $K_{\tau}$ be the long virtual knot consisting of these virtual crossings, the classical crossings in $\cup_i r(x_i)$, the arcs $a_{i,j,k,l}$, the arc $a_0$, the arc $a_1$, the interval $(-\infty,p_0)$, and the interval $(p_1,\infty)$. Define $\mathscr{F}^{-1}_0(\Sigma,\tau)=K_{\tau}$. The following theorem shows that this is well-defined on stability classes of long knot diagrams of surfaces.
 
\begin{proposition} \label{oto_diag_on_surf} $\mathscr{F}_0$ descends to a one-to-one correspondence $\mathscr{F}_0:\mathscr{K} \to \mathscr{K}(\mathscr{S})$ between equivalence classes of long virtual knots and stability classes of long knot diagrams on surfaces.
\end{proposition}
\begin{proof} As in \cite{kamkam}, the map $\mathscr{F}_0$ is well-defined. Since Reidemeister moves are defined within a small ball, the same proof applies to long virtual knots as well.
\newline
\newline
The assignment $\mathscr{F}_0^{-1}(\Sigma,\tau)=K_{\tau}$ is independent of the placement of arcs $a_{i,j,k,l}$, $a_0$, and $a_1$.  Indeed, any two ways to place arcs will produce coinciding Gauss diagrams on $\mathbb{R}$. Hence the resulting long virtual knots will be equivalent by only detour moves \cite{GPV}.  Also using a Gauss diagram argument, we see that if $(\Sigma_1,\tau_1)$ and $(\Sigma_2,\tau_2)$ are elementary equivalent, then $\mathscr{F}_0^{-1}(\Sigma_1,\tau_1)$ and $\mathscr{F}^{-1}_0(\Sigma_2,\tau_2)$ are equivalent by a classical Reidemeister move. Hence $\mathscr{F}_0^{-1}$ is well-defined. 
\newline
\newline
The construction of $\mathscr{F}_0^{-1}$ and $\mathscr{F}_0$ imply that $\mathscr{F}_0 \mathscr{F}^{-1}_0$ and $\mathscr{F}^{-1}_0\mathscr{F}_0$ are both the identity.
\end{proof}

\begin{proposition} \label{oto_thick_surf} There is a one-to-one correspondence $\mathscr{F}: \mathscr{K} \to \mathscr{K}(\mathscr{S})$ between equivalence classes of long virtual knots and stability classes of long knots in thickened surfaces.
\end{proposition}
\begin{proof} This follows by a nearly identical argument to the case of virtual knots, so the proof is omitted (see, for example \cite{roots}).
\end{proof}

\subsection{Decompositions, Correspondence of Decompositions} \label{sec_decomp_corr} In the previous section, it was proved that there is a one-to-one correspondence $\mathscr{F}: \mathscr{K} \to \mathscr{K}(\mathscr{S})$ between long virtual knots and long knots in thickened surfaces modulo stabilization. It is now shown that if $\mathscr{F}^{-1}(\Sigma,\tau)$ is a linearly prime long virtual knot and $(\Sigma,\tau)=(\Sigma_1\#\Sigma_2,\tau_1\#\tau_2)$, then either $\mathscr{F}^{-1}(\Sigma_1,\tau_1)$ is the long trivial knot or $\mathscr{F}^{-1}(\Sigma_2,\tau_2)$ is long trivial knot. This section completes the proof of Lemma \ref{prime_corr}. We first need some definitions.
\newline
\newline
A \emph{decomposition} of a long virtual knot $K$, is a pair of long virtual knots $K_1$,$K_2$ such that $K \leftrightharpoons K_1 \# K_2$. $K_1$ is called the \emph{left part} of the decomposition and $K_2$ is called the \emph{right part} of the decomposition.
\newline
\newline
Let $(\Sigma,\tau)$ be a long knot in thickened surface.  Let $S$ be a vertically proper embedded disc $I \times I$ or annulus $S^1 \times I$. Then $S$ is two-sided in $\Sigma \times I$.  We say that $S$ is a \emph{decomposition surface} if $\tau$ intersects $S$ transversally in exactly one point and cutting $\Sigma \times I$ along $S$ gives exactly two connected components. Cutting $\Sigma \times I$ along $S$ gives two long knots in thickened surfaces $(\Sigma_1,\tau_1)$, $(\Sigma_2,\tau_2)$ such that $\tau(0)=\tau_1(0)$ and $\tau(1)=\tau_2(1)$. We will call $(\Sigma_1,\tau_1)$ the \emph{left part} and $(\Sigma_2,\tau_2)$ the \emph{right part}.
\newline
\newline
The following proposition shows that if a long knot admits a decomposition surface, then there is a destabilization of the long knot to a concatenation. It is thus sufficient to consider only those decompositions which are concatenations.

\begin{proposition} Let $S$ be a decomposition surface of $(\Sigma,\tau)$ having left component $(\Sigma_1,\tau_1)$ and right component $(\Sigma_2,\tau_2)$. Then $(\Sigma,\tau) \sim_s (\Sigma_1 \# \Sigma_2,\tau_1\#\tau_2)$.
\end{proposition}
\begin{proof} If $S$ is a disc $I \times I$, then the proposition follows immediately from the definition of concatenation. We suppose that $S$ is an annulus $S^1 \times I$. Since $S$ is two-sided, there is a neighborhood $N$ of $S$ homeomorphic to $S \times [-1 ,1]$ such that $S$ identified with $S \times \{0\}$. Take a closed sector $A$ of $S$ containing the intersection with $\text{im}(\tau)$. Choose an $\varepsilon$, $0<\varepsilon\le 1$, such that the subset $V \subset N$ which is homeomorphic to $\text{cl}(S\backslash A) \times [-\varepsilon,\varepsilon]$ has no intersections with $\text{im}(\tau)$. It follows that $\text{cl}(\partial V \backslash (\Sigma \times \partial I))$ is an annulus $W$ having no intersections with $\tau$. Destabilize along $W$ to get $(\Sigma',\tau')$. Then $A \approx I \times I$ is a decomposition surface of $(\Sigma',\tau')$ and $(\Sigma',\tau')=(\Sigma_1 \# \Sigma_2,\tau_1\#\tau_2)$. This proves the proposition (see Figure \ref{all_tang_ops}).   
\end{proof}

\begin{figure}[h]
\[
\begin{array}{ccc}
\scalebox{.5}{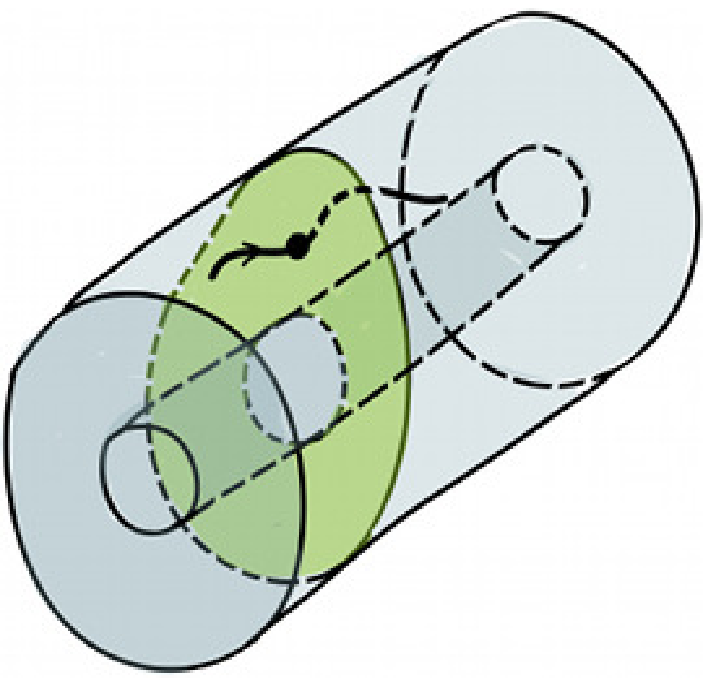} & \scalebox{.4}{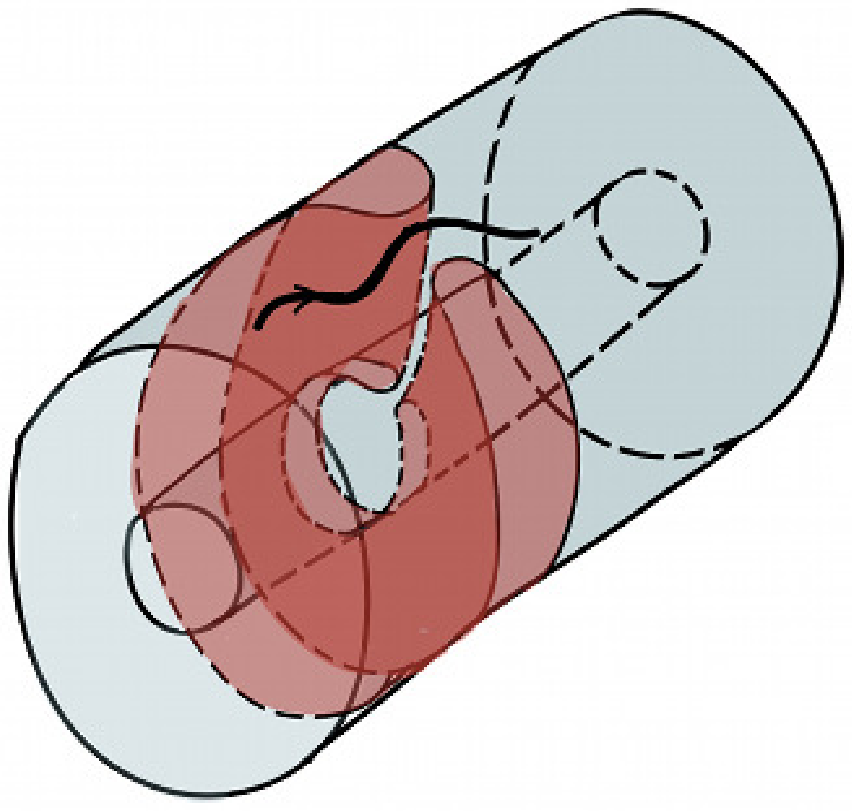} & 
\scalebox{.5}{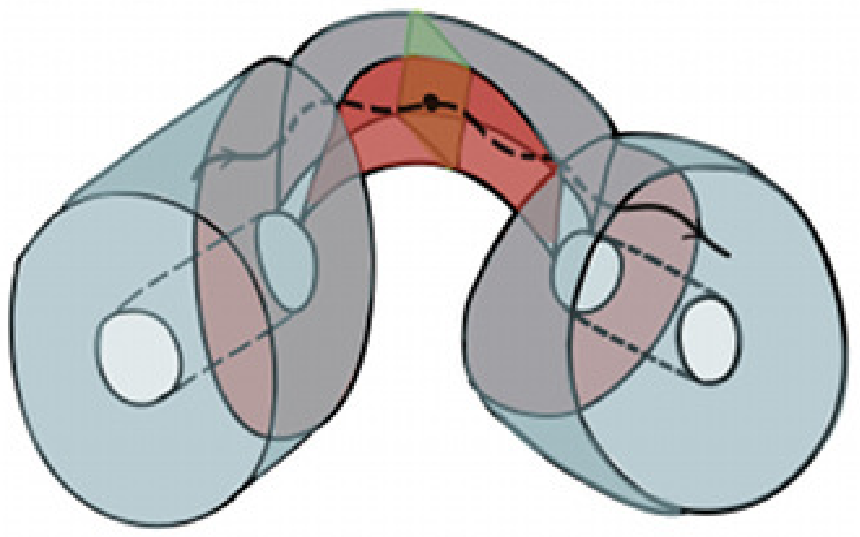}\\
\text{An annular decomposition.} & \text{A destabilization.} & \text{A concatention.} \\ 
\end{array}
\]
\caption{Every annular decomposition can be represented as a concatenation of long knots on thickened surfaces.} \label{all_tang_ops}
\end{figure}

\begin{proposition} \label{corr_decomp_thm} Suppose that $(\Sigma,\tau)$ is a long knot in a thickened surface and $(\Sigma,\tau)=(\Sigma_1\#\Sigma_2,\tau_1\#\tau_2)$. If $\mathscr{F}^{-1}(\Sigma,\tau)=K$ and $\mathscr{F}^{-1}(\Sigma_i,\tau_i)=K_i$ for $i=1,2$, then $K \leftrightharpoons K_1 \# K_2$.	
\end{proposition}
\begin{proof} By Theorem \ref{oto_thick_surf}, we may consider $(\Sigma,\tau)$ and $(\Sigma_i,\tau_i)$ as long knot diagrams on surfaces. Consider the construction of $K_{\tau}\leftrightharpoons K$. By Theorem \ref{oto_diag_on_surf}, all ways to make choices regarding placement of arcs and crossings give equivalent long virtual knots. Therefore, we may make the following choices. On $\mathbb{H}_L$, make the long virtual knot diagram corresponding to $f_L(\mathscr{F}_0^{-1}(\Sigma_1,\tau_1))$ (see Definition \ref{defn_concat_virtual}).  On $\mathbb{H}_R$, make the long virtual knot diagram corresponding to $f_R(\mathscr{F}_0^{-1}(\Sigma_2,\tau_2))$.  Adding in the point at the origin, we see that we have a long virtual knot diagram which is equivalent to both $K_{\tau}$ and $K_1 \# K_2$. 
\end{proof}

\begin{corollary}\label{prim_corr_fin} Let $(\Sigma,\tau)$ be a long knot in a thickened surface and suppose that $\mathscr{F}^{-1}(\Sigma,\tau)=K$ is linearly prime. If $(\Sigma,\tau)=(\Sigma_1\#\Sigma_2,\tau_1\#\tau_2)$, then either $\mathscr{F}^{-1}(\Sigma_1,\tau_1)\leftrightharpoons \to $ or $\mathscr{F}^{-1}(\Sigma_2,\tau_2)\leftrightharpoons \to$. 
\end{corollary}
\begin{proof} Suppose that $\mathscr{F}^{-1}(\Sigma_i,\tau_i)=K_i$ for $i=1,2$.  By Theorem \ref{corr_decomp_thm}, $K \leftrightharpoons K_1 \# K_2$. It follows that either $K_1 \leftrightharpoons \to$ or $K_2 \leftrightharpoons \to$.  
\end{proof}

\noindent The establishment of Corollary \ref{prim_corr_fin} completes the proof of Lemma \ref{prime_corr} from Section 1.3: Statement of Main Results.\hfill$\square$

\section{Geometric Constructions} \label{sec_geom_construct} 

\noindent In this section we define some geometric constructions which will be used in the proofs of Lemma \ref{main_1}, Lemma \ref{main_2}, and Theorem \ref{main_3}. After the short subsection \ref{prelim_obs}, the remaining subsections may be read independently of one another. The discussion is along the lines of traditional geometric topology, as in \cite{hempel}. In Section \ref{sec_bead}, we discuss the bead-on-string isotopy. It is the long knot in $\Sigma \times I$ analog of Conway's proof that classical long knots commute. In Section \ref{sec_vert_comps}, we briefly discuss a lemma which gives sufficient conditions for a vertically proper embedded disc to intersect exactly one component of $(\partial \Sigma) \times I$. In Section \ref{sec_bunches}, we prove some lemmas for reducing the number of connected components in the intersection of two vertically proper embedded surfaces.

\subsection{Preliminary Observations} \label{prelim_obs} Let $\Sigma\in\mathscr{S}$ and consider $M=\Sigma\times I$ as a $3$-manifold. Since $\Sigma \ne S^2$, it follows that $M$ is irreducible. That is, any $2$-sphere in $M$ bounds a $3$-ball. However, we note that $M$ will generally not be boundary irreducible. 

\subsection{Bead-On-A-String Isotopy} \label{sec_bead} The following lemma is the version of Conway's proof of the commutativity of classical knots which applies to long knots in thickened surfaces. It is used frequently in the proofs of Lemma \ref{main_2} and Theorem \ref{main_3}.

\begin{lemma}[Bead-On-A-String Isotopy] \label{beadonastring}Let $\Sigma \in \mathscr{S}$ and let $K$ be a long knot in $\Sigma \times I$.  Let $N$ be any closed regular neighborhood of a distinguished boundary component of $\Sigma$ and let $T$ be any closed regular neighborhood of $K$.  Suppose that there is a closed $3$-ball $B$ in $\Sigma \times I$ such that $\partial B$ has exactly two intersections with $K$, both of which are transversal. Then there is a (p.l) ambient isotopy $F:(\Sigma \times I) \times I \to \Sigma \times I$ of $K$ such that $F(B,1) \subset (N \times I) \cap T$ and such that $F_t|_{\text{cl}((B\cup T)^c)}$ is the indentity.
\end{lemma}

\begin{proof}(Sketch) First note that there is a (p.l) ambient isotopy of $\Sigma \times I$ which sends $K \cap B$ into $B \cap T$ and which is the identity outside of $B$.  We may then consider $T$ as an embedded solid cylinder $D^2 \times I$ in $\Sigma \times I$ containing a long classical knot $K'$. Then there is a (p.l.) ambient isotopy of the solid cylinder which is the identity on the boundary of the cylinder and sends $T \cap B \cap K'$ into $(N \times I) \cap T$. Composing these together gives the desired (p.l.)   ambient isotopy.
\end{proof}

\subsection{Vertically Proper Embedded discs and Components of $(\partial \Sigma) \times I$} \label{sec_vert_comps} The opposite sides of a vertically proper embedded disc $R\approx I \times I$ in $\Sigma \times I$, may intersect either one or two components of $(\partial \Sigma) \times I$. The following lemma gives sufficient conditions for the number of components that $R$ intersects to be equal to one. The lemma and its contrapositive are used frequently in the remainder of the text. The proof is elementary and is therefore omitted.

\begin{lemma}\label{lemma_disc_genus_red} Let $\Sigma \in \mathscr{S}$ and let $(\Sigma,\tau)$ be a long knot in $\Sigma \times I$. If $R$ is a vertically proper embedded disc in $\Sigma \times I$ which is disconnecting or genus reducing, then there is unique a connected component $C$ of $(\partial \Sigma) \times I$ such that $R \cap (\partial \Sigma \times I) \subset C$.
\end{lemma}

\subsection{Bunches} \label{sec_bunches} The proofs of Theorem \ref{main_3} and Lemmas \ref{main_1} and \ref{main_2} all rely upon reducing the number of connected components of intersections of vertically proper embedded discs and annuli (as defined above).  In this section we introduce terminology for the type of intersections that can occur.  In addition, we describe several geometric constructions that will be used throughout.

\begin{definition}[Bunches on a disc] Let $f: I \times I \to R$ be a vertically proper embedded disc in $\Sigma \times I$. Let $e:I \to R$ be a (p.l.) embedded interval.  Then $e$ is \emph{vertical} if the image of one endpoint of $I$ is in $\Sigma \times \{0\}$ and the image of the other endpoint of $I$ is in $\Sigma \times \{1\}$. The interval $e$ is \emph{horizontal} if the image of one endpoint of $I$ is in $f(\{0\} \times I) \subset R$ and the image of the other endpoint of $I$ is in $f(\{1\} \times I) \subset R$. The embedded interval $e$ is \emph{cornered} if the images of the endpoints of $I$ are in non-opposite sides of $R$. If $e$ is cornered, the \emph{corner} $c_e$ of $e$ is the point of intersection of the adjacent sides of $R$ which $e$ intersects. The embedded interval $e$ is \emph{partisan} if its endpoints $e(0) \ne e(1)$ are on the same side of $R$. A \emph{bunch} of embedded intervals is a finite (possibly empty) collection of pairwise non-intersecting embedded intervals which are either all vertical, all horizontal, all cornered, or all partisan. 
\end{definition}

\begin{definition}[Bunches on an Annulus] Let $f:S^1 \times I \to R$ be a vertically proper embedded annulus in $\Sigma \times I$. A (p.l.) embedded interval $e:I \to R$ is said to be \emph{vertical} if $e(0)$ and $e(1)$ are in different components of $\partial R$.  A (p.l.) embedded interval $e:I \to R$ is said to be \emph{partisan} of $e(0)$ and $e(1)$ are in the same component of $\partial R$. A (p.l.) embedded circle $c:S^1 \to R$ is said to be \emph{horizontal} if it does not intersect $\partial R$ and does not bound a disc in $R$. A \emph{bunch} is a set containing exclusively pairwise non-intersecting vertical intervals, exclusively horizontal circles, or exclusively partisan intervals. 
\end{definition}

\begin{remark} We will frequently abuse notation and fail to distinguish between an embedding of an interval $e:I \to R$ and the image of the embedding. In this case, the \emph{endpoints} of $e$ will be taken to mean the points $e(0)$ and $e(1)$ in $R$.
\end{remark}

\begin{definition}[Inessential circles] Let $R$ be a vertically proper embedded annulus or disc.  Let $c:S^1 \to R$ be a (p.l.) embedded circle not intersecting $\partial R$.  Suppose that $\text{im}(c)$ bounds a disc $D_c$ on $R$.  Then $c$ is called an \emph{inessential circle}.
\end{definition}

\subsubsection{Inessential Circles} Let $\Sigma \in \mathscr{S}$. Let $R_1$, $R_2$ be vertically proper embedded surfaces such that for $i=1,2$, $R_i$ is a vertically proper embedded disc or annulus. Suppose that $R_1 \cap R_2$ contains an embedded circle $c$. The following lemmas detail how to decrease the number of connected components of $R_1 \cap R_2$ in the two cases in which inessential circles appear.

\begin{lemma} \label{two_inessential} Suppose $c$ is an inessential circle in both $R_1$ and $R_2$, and let $D_1$, $D_2$ be discs such that $\text{im}(c)=\partial D_1=\partial D_2$ where $D_i \subset R_i$ for $i=1,2$. Suppose that $c$ is innermost in the sense that $D_1 \cup D_2$ contains no connected components of $R_1 \cap R_2$ other than $c$.  Then $D_1 \cup D_2$  is a $2$-sphere which bounds a $3$-ball $B_{12}$ in $\Sigma \times I$. Moreover, for every neighborhood $U$ of $B_{12}$ and $i \in \{1,2\}$, there is a vertically proper embedded surface $R_{i}'$ and an ambient isotopy $F: (\Sigma \times I) \times I \to \Sigma \times I$ such that $F_1(R_{i})=R_i'$, $F_t|_{\Sigma \times I \backslash U}=\text{id}_{\Sigma \times I \backslash U}$ for all $t \in I$, and $R_i' \cap R_{3-i}$ has fewer connected components than $R_1 \cap R_2$.
\end{lemma}
\begin{proof} Certainly, $D_1 \cup D_2$ is a 2-sphere.  Since $\Sigma \times I$ is irreducible, it must bound a 3-ball $B_{12}$ (see Figure \ref{two_inessential_fig}). Let a neighborhood $U$ of $B_{12}$ in $\Sigma \times I$ be given. For $i=1,2$, $D_{3-i}$ is a compression disc for $R_{i}$ such that compressing $R_{i}$ along $D_{3-i}$ in $U$ gives two surfaces (see Lemma 4.6 of \cite{hempel}). One is a $2$-sphere and the other is a disc. Setting $R_{i}'$ to be the disc component after compression proves the lemma.  
\end{proof}

\begin{figure}[h]
\[
\begin{array}{cc}
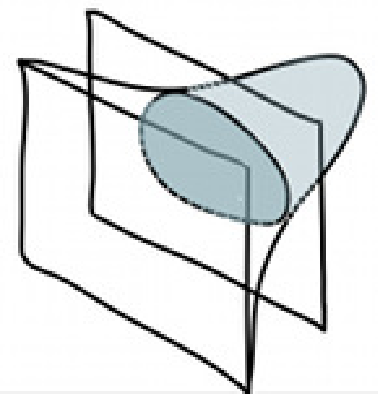 & 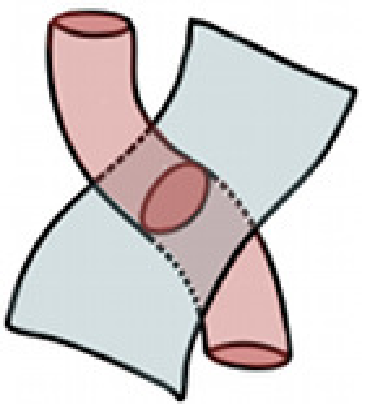
\end{array}
\]
\caption{An inessential circle in $R_1$ and $R_2$ (left) and an inessential circle in $R_1$ that is horizontal circle in $R_2$(right).} \label{two_inessential_fig}
\end{figure}

\begin{lemma}\label{cutsoffthreeball} Suppose that $R_2$ is an annulus, $c$ is inessential in $R_1$, and $c$ is horizontal in $R_2$. Suppose also that $c$ is innermost in the sense that the disc $D\subset R_1$ which $c$ bounds contains no components of $R_1 \cap R_2$ other than $c$. Then $R_2$ cuts off a $3$-ball $B_{12}$ in $\Sigma \times I$.  
\end{lemma}
\begin{proof} It is easy to see that either $R_2$ cuts off a $3$-ball or $R_2$ is injective. If $R_2$ is injective, the fact that it is two-sided then implies that it is also incompressible (see \cite{matveevbook}). Hence, the hypothesis that $c$ bounds a disc $D$ on $R_1$ such that $D \cap R_2=\text{im}(c)$ implies that $c$ bounds a disc on $R_2$. This contradicts the fact that $c$ is horizontal in $R_2$.       
\end{proof}

\subsubsection{Partisan Bunches} Let $\Sigma \in \mathscr{S}$. Let $R_1$, $R_2$ be vertically proper embedded surfaces such that for $i=1,2$, $R_i$ is a vertically proper embedded disc or annulus. Suppose that $R_1 \cap R_2$ contains a partisan interval $e$. The following lemmas detail how to decrease the number of connected components of $R_1 \cap R_2$ in the two cases in which partisan intervals appear.

\begin{lemma} \label{partisan_top} Suppose that the endpoints of the partisan interval $e$ are both in $\Sigma \times \{t\}$, where $t\in\{0,1\}$. For $i=1,2$, there is a unique interval $a_i$ in $\partial R_i \cap (\Sigma \times \{t\})$ such that $\text{im}(e) \cup \text{im}(a_i)$ bounds a disc $D_i$ on $R_i$. Suppose that $e$ is innermost in the sense that $D_1 \cup D_2$ contains no components of $R_1 \cap R_2$ other than $e$.  Then the following hold:
\begin{enumerate}
\item $\text{im}(a_1) \cup \text{im}(a_2)$ bounds a disc $D_3 \subset \Sigma \times \{t\}$.
\item $D_1 \cup D_2 \cup D_3$ is a $2$-sphere bounding a 3-ball $B_{12}$ in $\Sigma \times I$.
\item For every neighborhood $U$ of $B_{12}$ and $i \in \{1,2\}$, there is a vertically proper embedded surface $R_{i}'$ and an ambient isotopy $F: (\Sigma \times I) \times I \to \Sigma \times I$ such that $F_1(R_{i})=R_i'$, $F_t|_{\Sigma \times I \backslash U}=\text{id}_{\Sigma \times I \backslash U}$ for all $t \in I$, and $R_i' \cap R_{3-i}$ has fewer connected components than $R_1 \cap R_2$.
\end{enumerate}
\end{lemma}
\begin{proof} Note that $D_1 \cup D_2$ is a disc such that $\partial (D_1 \cup D_2) \subset \Sigma \times \{t\}$. Since $\partial (D_1 \cup D_2)$ is contractible in $\Sigma \times I$, we can find a contraction of $\partial (D_1 \cup D_2)$ to a point by applying the projection $\Sigma \times I \to \Sigma$.  It follows that $\text{im}(a_1) \cup \text{im}(a_2)=\partial (D_1 \cup D_2)$ bounds a disc on $\Sigma \times \{t\}$. It is easy to compute that $D_1 \cup D_2 \cup D_3$ is a 2-sphere. Since $\Sigma \times I$ is irreducible, $D_1 \cup D_2 \cup D_3$ must bound a 3-ball $B_{12}$ (see the left hand side of Figure \ref{partisan_fig}). Since $D_1 \cup D_2$ contains no components of $R_1 \cap R_2$ other than $e$, it follows that $D_{3-i}$ is a boundary compression disc of $R_{i}$ for $i=1,2$. Let a neighborhood $U$ of $B_{12}$ in $\Sigma \times I$ be given. The compression of $R_i$ along $D_{3-i}$ in $U$ gives two surfaces (see Lemma 4.7 of \cite{hempel}). The two surfaces are both discs, one of which is a vertically proper embedded $I \times I$ and the other is not vertically proper embedded. Choose $R_i'$ to be the vertically proper embedded surface after compression on $D_{3-i}$.
\end{proof}

\begin{figure}[h]
\[
\begin{array}{ccc}
\scalebox{.6}{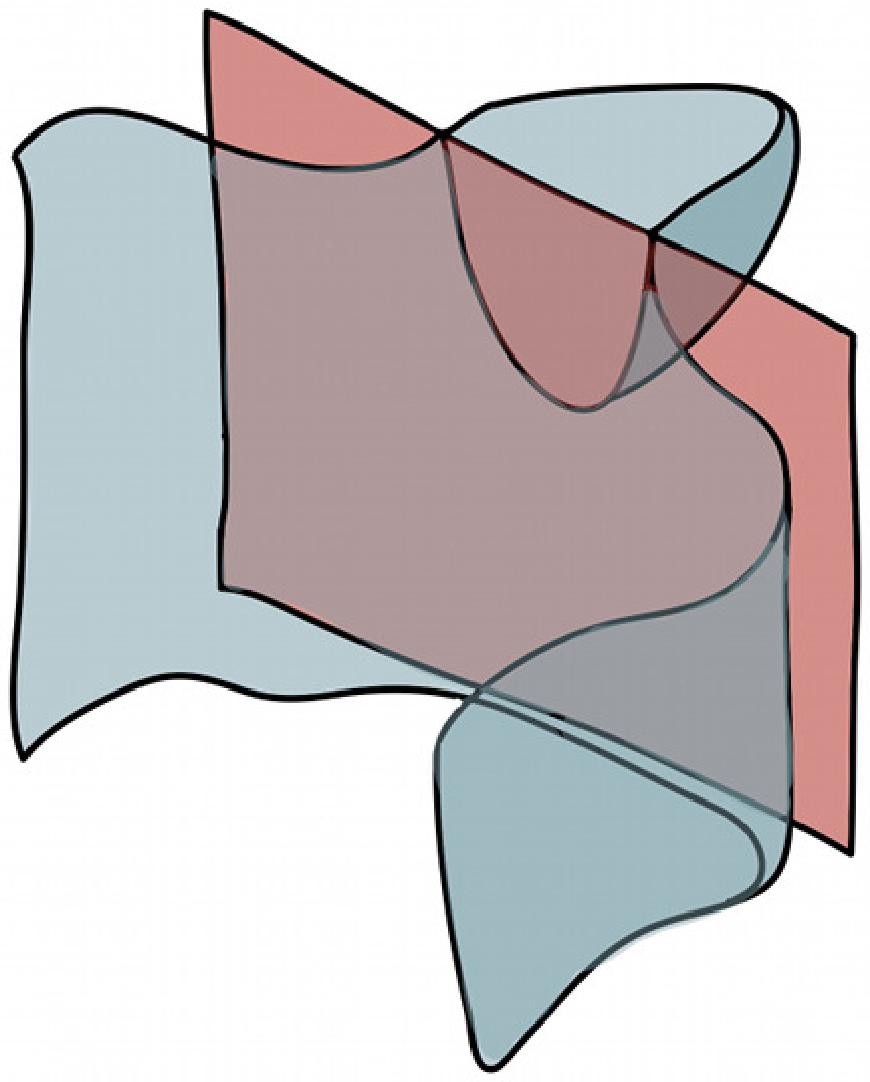} & \hspace{1cm} & \scalebox{.6}{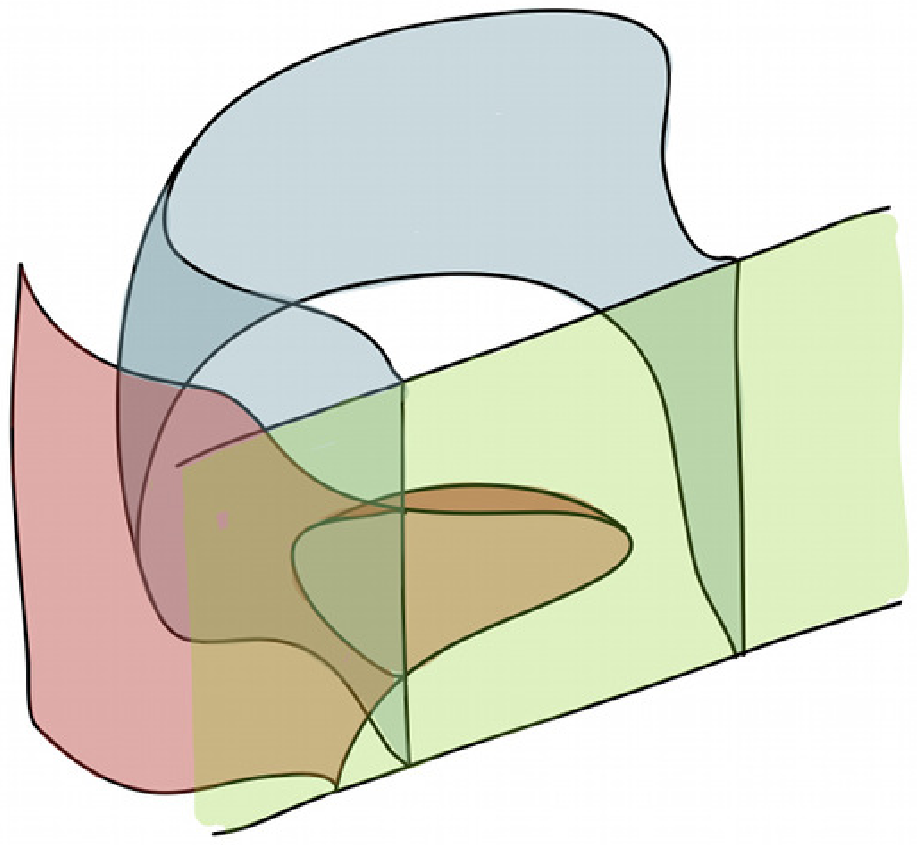}
\end{array}
\]
\caption{An innermost partisan interval intersecting the ``top'' as in Lemma \ref{partisan_top}(left) and an innermost partisan interval intersecting the ``side'' as in Lemma \ref{partisan_side}(right).} \label{partisan_fig}
\end{figure}

\begin{lemma} \label{partisan_side} Suppose that the endpoints of $e$ are both in a single component $C$ of $(\partial \Sigma) \times I$. Then $R_i\approx I \times I$ for $i=1,2$. Suppose that $R_1 \cap C$ consists of two disjoint embedded intervals. For $i=1,2$, there is a unique sub-interval $a_i$ in $R_i \cap C$ between the endpoints of $e$. Then for $i=1,2$, $\text{im}(a_i) \cup \text{im}(e)$ is a 1-sphere bounding a disc $D_i \subset R_i$. Suppose that $e$ is innermost in the sense that $D_1 \cup D_2$ contains no connected components of $R_1 \cap R_2$ other than $e$. Then the following hold:
\begin{enumerate}
\item $\text{im}(a_1) \cup \text{im}(a_2)$ bounds a disc $D_3 \subset C$.
\item $D_1 \cup D_2 \cup D_3$ is a $2$-sphere bounding a 3-ball $B_{12}$ in $\Sigma \times I$.
\item For every neighborhood $U$ of $B_{12}$ and $i \in \{1,2\}$, there is a vertically proper embedded disc $R_{i}'$ and an ambient isotopy $F: (\Sigma \times I) \times I \to \Sigma \times I$ such that $F_1(R_{i})=R_i'$, $F_t|_{\Sigma \times I \backslash U}=\text{id}_{\Sigma \times I \backslash U}$ for all $t \in I$, and $R_i' \cap R_{3-i}$ has fewer connected components than $R_1 \cap R_2$.
\end{enumerate}
\end{lemma}
\begin{proof} Since $R_1 \cap C$ is two disjoint intervals and $D_1 \cup D_2$ contains no connected components of $R_1 \cap R_2$ other than $e$, we cannot have that $\text{im}(a_1) \cup \text{im}(a_2)$ is homotopically non-trivial in $C$. Hence, $\text{im}(a_1) \cup \text{im}(a_2)$ bounds a disc $D_3 \subset C$. Then $D_1 \cup D_2 \cup D_3$ is a 2-sphere bounding a 3-ball $B_{12}$ (see the right hand side of Figure \ref{partisan_fig}). Since $D_1 \cup D_2$ contains no connected components of $R_1 \cap R_2$ other than $e$, it follows that $D_{3-i}$ is a boundary compression disc of $R_{i}$ for $i=1,2$. Let a neighborhood $U$ of $B_{12}$ in $\Sigma \times I$ be given. Compressing $R_{i}$ along $D_{3-i}$ in $U$ (see Lemma 4.7 of \cite{hempel}) gives a surface with two components, one of which is a vertically proper embedded disc and one of which is a disc which is not vertically proper embedded. Choose $R_i'$ to be the vertically proper embedded disc after compression.
\end{proof}

\subsubsection{Cornered Bunches} Let $\Sigma \in \mathscr{S}$. Let $R_1$, $R_2$ be vertically proper embedded discs $I \times I$. Suppose that $R_1 \cap R_2$ contains a cornered interval $e$. The following lemma details how to decrease the number of connected components of $R_1 \cap R_2$ when an innermost cornered interval appears. 

\begin{lemma} \label{corner_lemma} Let $c_i$ be the corner of $e$ on $R_i$ for $i=1,2$. Suppose that $e$ has an endpoint in $\Sigma \times \{t\}$, where $t \in \{0,1\}$, and an endpoint on a component $C$ of $(\partial \Sigma) \times I$. There is a unique interval $a_i$ between the endpoint of $e$ on $R_i \cap (\Sigma \times \{t\})$ and $c_i$ and there is a unique interval $b_i$ between the endpoint of $e$ on $R_i \cap C$ and $c_i$. Then $\text{im}(a_i) \cup \text{im}(b_i) \cup \text{im}(e)$ bounds a disc $D_i$ on $R_i$. Suppose that $e$ is innermost in the sense that $D_1 \cup D_2$ contains no components of $R_1 \cap R_2$ other than $e$. Let $E$ be the connected component of $C\backslash R_1$ containing $\text{int}(\text{im}(b_2))$ and let $E_{12}$ the connected component of $E\backslash \text{im}(b_2)$ which does not intersect $\Sigma \times \{1-t\}$.  Let $e_{12}$ be the interval on $E_{12} \cap (\Sigma \times \{t\})$ between $c_1$ and $c_2$, so that $\partial E_{12}=\text{im}(b_1) \cup \text{im}(b_2) \cup \text{im}(e_{12})$. Then the following hold:
\begin{enumerate}
\item $\text{im}(e_{12}) \cup \text{im}(a_1) \cup \text{im}(a_2)$ bounds a disc $D_3 \subset \Sigma \times \{t\}$.
\item $D_1 \cup D_2 \cup D_3 \cup E_{12}$ is a 2-sphere bounding a 3-ball $B_{12}$ in $\Sigma \times I$.
\item For every neighborhood $U$ of $B_{12}$ and $i \in \{1,2\}$, there is a vertically properly embedded disc $R_{i}'$ and an ambient isotopy $F: (\Sigma \times I) \times I \to \Sigma \times I$ such that $F_1(R_{i})=R_i'$, $F_t|_{\Sigma \times I \backslash U}=\text{id}_{\Sigma \times I \backslash U}$ for all $t \in I$, and $R_i' \cap R_{3-i}$ has fewer connected components than $R_1 \cap R_2$.
\end{enumerate}  
\end{lemma}
\begin{proof} For the first claim, note that $\text{im}(e_{12}) \cup \text{im}(a_1) \cup \text{im}(a_2)$ is contractible in $\Sigma \times I$. Applying the projection $\Sigma \times I \to \Sigma$ to this contraction shows that $\text{im}(e_{12}) \cup \text{im}(a_1) \cup \text{im}(a_2)$ is contractible in $\Sigma$.  Thus, $\text{im}(e_{12}) \cup \text{im}(a_1) \cup \text{im}(a_2)$ bounds a disc $D_3 \subset \Sigma \times \{t\}$. It follows that $D_1 \cup D_2 \cup D_3 \cup E_{12}$ is a $2$-sphere.  Since $\Sigma \times I$ is irreducible, the $2$-sphere must bound a $3$-ball $B_{12}$ (see Figure \ref{corner_fig}). Since $e$ is innermost, $D_{3-i}$ is a boundary compression disc for $R_{i}$ for $i=1,2$. Let a neighborhood $U$ of $B_{12}$ in $\Sigma \times I$ be given. Compressing $R_{i}$ along $D_{3-i}$ in $U$ (see Lemma 4.7 of \cite{hempel}) gives a surface with two components: one of which is a vertically proper embedded disc and one of which is a disc which is not vertically proper embedded. Choose $R_i'$ to be the disc which is vertically proper embedded disc after compression on $D_{3-i}$.
\end{proof}

\begin{figure}[h]
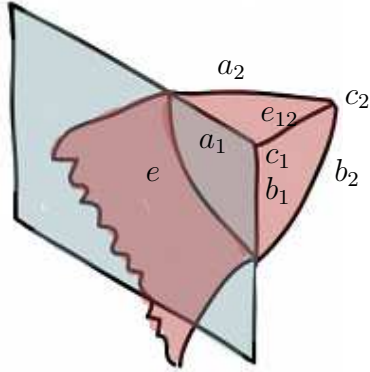
\caption{An Innermost Cornered Interval.} \label{corner_fig}
\end{figure}

\subsubsection{Horizontal Bunches} Let $\Sigma \in \mathscr{S}$. Let $R_1$, $R_2$ be vertically proper embedded discs $I \times I$. Suppose that $R_1 \cap R_2$ is a horizontal bunch. We number the intervals in the bunch from ``bottom'' to ``top'' on $R_i$, $e_i^1$, $e_i^2$,$\ldots$, $e_i^{n}$. By ``bottom'' to ``top'', we mean that cutting along $e_i^1$ gives two discs, one of which contains no other $e_i^k$ but contains $R_i \cap (\Sigma \times \{0\})$. Moreover, cutting along $e_i^j$ and $e_i^k$ contains exactly those intervals $e_i^s$ for which $i \le s \le k$. We identify $e_i^0$ with $R_i\cap (\Sigma \times \{0\})$ and $e_i^{n+1}$ with $R_i \cap (\Sigma \times \{1\})$. For $i=1,2$, the intervals $e_i^1,\ldots, e_i^n$ divide $R_i$ into rectangles $r_i^1,\ldots, r_i^{n+1}$ such that $\text{im}(e_i^{k-1}) \cup \text{im}(e_i^{k}) \subset \partial r_i^k$ (see Figure \ref{horiz_int_rect_fig}).
\newline
\newline
\noindent The following lemma shows how to decrease the number of connected components of $R_1 \cap R_2$ when horizontal intervals occur (see also Lemma 4.7 of \cite{hempel}).

\begin{figure}[htb]
\scalebox{.75}{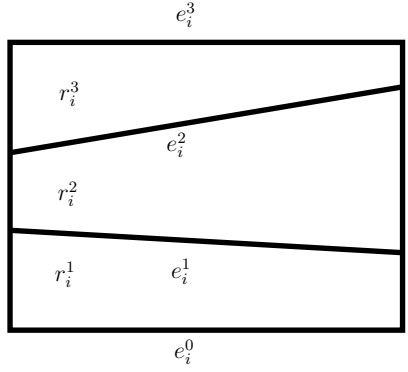}
\caption{Horizontal bunches divide $R_i$ into rectangles.} \label{horiz_int_rect_fig}
\end{figure}

\begin{lemma}\label{horiz_bunches} Suppose there is exactly one component $C$ of $(\partial \Sigma) \times I$ for which $R_1 \cap C \ne \emptyset$. Then the following hold:
\begin{enumerate}
\item There are rectangles $r_i^j$ and $r_{3-i}^k$ such that $r_i^j \cap R_{3-i}=r_{3-i}^k \cap R_i$. The $j,k$ may be chosen so that either $j=k=1$, $j=k=n+1$, or $1<j,k<n+1$.
\item If $j=k=1$ or $j=k=n+1$ then $(r_i^j \cup r_{3-i}^k) \cap C$ consists of two intervals cutting off a pair of disjoint discs $D_1$, $D_2$ in $C$ such that the following hold:
\begin{enumerate}
\item $\partial (D_1 \cup D_2 \cup r_i^j \cup r_{3-i}^k)$ is a $1$-sphere which bounds a disc $D_3 \subset \Sigma \times \partial I$.
\item $D_1 \cup D_2 \cup D_3 \cup r_i^j \cup r_{3-i}^k$ is a 2-sphere which bounds a 3-ball $B_{12}$ in $\Sigma \times I$.
\item For every neighborhood $U$ of $B_{12}$ and $i \in \{1,2\}$, there is a vertically proper embedded disc  $R_{i}'$ and an ambient isotopy $F: (\Sigma \times I) \times I \to \Sigma \times I$ such that $F_1(R_{i})=R_i'$, $F_t|_{\Sigma \times I \backslash U}=\text{id}_{\Sigma \times I \backslash U}$ for all $t \in I$, and $R_i' \cap R_{3-i}$ has fewer connected components than $R_1 \cap R_2$.
\end{enumerate}
\item If $1<j,k<n+1$, then $(r_i^j \cup r_{3-i}^k) \cap C$ consists of two $1$-spheres which bound two discs $D_1$ and $D_2$ contained in $C$ such that the following hold.
\begin{enumerate}
\item $D_1 \cup D_2 \cup r_i^j \cup r_{3-i}^k$ is a 2-sphere which bounds a 3-ball $B_{12}$ in $\Sigma \times I$.
\item For every neighborhood $U$ of $B_{12}$ and $i \in \{1,2\}$, there is a vertically proper embedded disc $R_{i}'$ and an ambient isotopy $F: (\Sigma \times I) \times I \to \Sigma \times I$ such that $F_1(R_{i})=R_i'$, $F_t|_{\Sigma \times I \backslash U}=\text{id}_{\Sigma \times I \backslash U}$ for all $t \in I$, and $R_i' \cap R_{3-i}$ has fewer connected components than $R_1 \cap R_2$.
\end{enumerate} 
\end{enumerate} 
\end{lemma}
\begin{proof} Consider the first claim.  If $n=1$, $R_1$ and $R_2$ each have two sub-rectangles and the claim is obviously true.  Suppose then that $n \ge 2$.  We will argue by way of contradiction.  Suppose that no such pair of sub-rectangles exists. Choose any $r_1^i$ such that $1<i<n+1$. Then $r_1^i \cap R_2=\text{im}(e_2^a)\cup\text{im}(e_2^b)$, $a<b$, where $b-a>1$. We may suppose that $r_1^i$ is chosen so that $b-a$ is as small as possible. Consider $A=r_1^i \cup \bigcup_{a<l\le b} r_2^l$. Then $A$ is an annulus or a M\"{o}bius strip.
\newline
\newline
Cut $C$ along $R_2 \cap C$.  This gives two discs $M$ and $N$. Now $r_1^i \cap C$ consists of two intervals $u_1^i$ and $v_1^i$.  Suppose without loss of generality that $M$ contains $\text{int}(u_1^i)$. Let $p_2^i$ be the interval in $A \cap R_2 \cap C$ between the endpoints of $u_1^i$. Since $u_1^i$ intersects $R_2\cap C$ only in its endpoints, it follows that $\text{im}(u_1^i)\cup \text{im}(p_2^i)$ is a 1-sphere which bounds a disc $E_1 \subset M$. Similarly, there is an interval $q_2^i$ in $A \cap R_2\cap C$ between the endpoints of $v_1^i$ such that $\text{im}(v_1^i)\cup \text{im}(q_2^i)$ bounds a disc $E_2$ in $M$ or in $N$. Since $\partial A$ contains more than one boundary component, $A$ is not a M\"{o}bius strip. Hence, $A \cup E_1 \cup E_2$ must be a $2$-sphere. Since $\Sigma \times I$ is irreducible, it bounds a $3$-ball $B$. Thus, if there is any rectangle $r_1^s$ which intersects $B$, we must have $\partial r_1^s \subset \partial B$. Moreover, $e_1^{s-1}$ and $e_1^{s}$ correspond to some intervals $e_2^c$ and $e_2^d$ contained in $A \cap C \cap R_2$. By construction, we must have that $|d-c|<|a-b|$. Since $r_1^i$ was chosen so that $|a-b|$ was minimal, this is a contradiction. This establishes the first claim.  
\newline
\newline
Consider now the second claim of the lemma. The existence of the discs $D_1$ and $D_2$ is transparent from the definitions.  Note that $(D_1 \cup D_2 \cup r_i^j \cup r_{3-i}^k) \cap (\Sigma \times \partial I)$ is a $1$-sphere $Y$ that is contractible in $\Sigma \times I$. Composing this contraction with the projection $\Sigma \times I \to \Sigma$ shows that $Y$ is contractible in $\Sigma$. Thus, $Y=\partial D_3$ for some disc $D_3$ in $\Sigma \times (\partial I)$. It follows that $D_1 \cup D_2 \cup D_3 \cup r_i^j \cup r_{3-i}^j$ is a $2$-sphere bounding a $3$-ball $B_{12}$ (see the left hand side of Figure \ref{horiz_bunch_fig}). Since $r_i^j \cup r_{3-i}^j$ contains no other components of $R_1 \cap R_2$ other than $r_i^j \cap r_{3-i}^k$, we have that $r_i^j$ is a boundary compression disc for $R_{3-i}$ and $r_{3-i}^k$ is a boundary compression disc for $R_i$. Let a neighborhood $U$ of $B_{12}$ in $\Sigma \times I$ be given. Compressing $R_{3-i}$ along $r_i^j$ (or $R_i$ along $r_{3-i}^k$) in $U$ gives a surface of two components: one is a disc ``close'' to $ r_i^j \cup r_{3-i}^k$ while the other is a disc ``close'' to $(R_{3-i} \backslash r_{3-i}^k) \cup r_i^j$ (resp. $(R_{i} \backslash r_{i}^j) \cup r_{3-i}^k$). Only the second of these options is vertically proper embedded. Choose $R_i'$ to be the vertically proper embedded disc after compression.
\newline
\newline
Lastly, consider the third claim of the lemma statement. The first assertion and the assertion that $D_1 \cup D_2 \cup r_i^j \cup r_{3-i}^k$ is a $2$-sphere bounding a $3$-ball follow exactly as in the proof parts (1) and (2) of this lemma. Now consider the vertically proper embedded disc $(R_{3-i} \backslash r_{3-i}^k) \cup  r_i^j$. We may push this surface slightly off of $r_{i}^j$ to obtain a surface $R_{3-i}'$ such that $R_{3-i}' \cap R_i$ has fewer connected components than $R_1 \cap R_2$. Since $D_1 \cup D_2 \cup r_i^j \cup r_{3-i}^k$ bounds a $3$-ball, there is an ambient isotopy taking $R_{3-i}$ and $R_{3-i}'$. A similar construction may be used to determine the surface $R_i'$.
\end{proof}

\begin{figure}[h]
\[
\begin{array}{ccc}
\scalebox{.75}{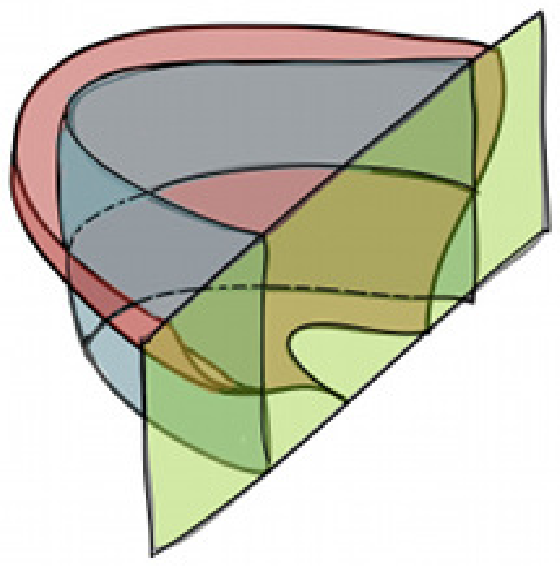} & \hspace{1cm} & \scalebox{.75}{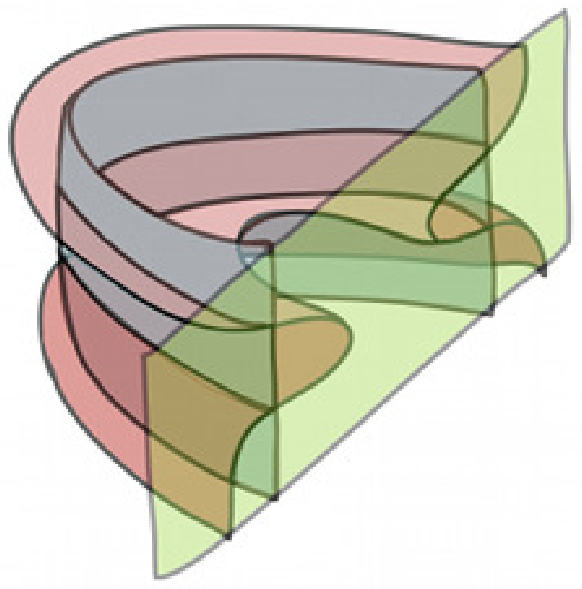}
\end{array}
\]
\caption{A horizontal bunch with only one interval (left) and a horizontal bunch with at least two intervals (right).} \label{horiz_bunch_fig}
\end{figure}

\section{Proofs of Lemma \ref{main_1}, Lemma \ref{main_2}, and Theorem \ref{main_3}} 
\subsection{Lemma \ref{main_1}: Kuperberg's Theorem-Long Knot Version} \label{sec_kup} In the present section, we prove the long knot version of Kuperberg's Theorem which was stated in the introduction (Lemma \ref{main_1}). We note that there are other generalizations of Kuperberg's theorem. For example, there is a generalization of Kuperberg's Theorem to twisted $I$-bundles \cite{bourgoin}. 

\begin{proof}[Proof of Lemma \ref{main_1}] Let $(\Sigma,\tau)$ be a long knot in $\Sigma \times I$.  We associate to $(\Sigma,\tau)$ the pair $(g(\Sigma),\#(\Sigma))$, where $g(\Sigma)$ denotes the genus and $\#(\Sigma)$ denotes the number of distinguished boundary components. For such pairs we give the dictionary ordering. If $(\Sigma,\tau)$ is not irreducible, then it must have a num reducing destabilization or a genus reducing destabilization.  Let $(\Sigma',\tau')$ be the result of this destabilization. Then $(g(\Sigma'),\#(\Sigma'))<(g(\Sigma),\#(\Sigma))$. Since the genus is bounded below by zero and the number of distinguished boundary components is either one or two, it follows that every $(\Sigma,\tau)$ which is not irreducible has a irreducible descendant. This proves the first claim.
\newline
\newline
For claims (2)-(4), we will suppose (as in \cite{kuperberg}) that there is a $(\Sigma_0,\tau_0)$ having two descendants $(\Sigma_0',\tau_0')$ and $(\Sigma_0'',\tau_0'')$ which do not have a common descendant. From there we will arrive at a contradiction. Let $F'$ be the vertically proper embedded surface of the destabilization of $(\Sigma_0,\tau_0)$ to $(\Sigma_0',\tau_0')$ and $F''$ the vertically proper embedded surface in the destabilization of $(\Sigma_0,\tau_0)$ to $(\Sigma_0'',\tau_0'')$. By a general position argument, the surfaces $F',F''$ may be chosen to intersect transversely. 
\newline
\newline
First suppose that $F' \cap F''=\emptyset$.  Let $(\Sigma_0''',\tau_0''')$ be the result of destabilizing $(\Sigma_0,\tau_0)$ along $F'$ and $F''$. Then $(\Sigma_0',\tau_0')$ and $(\Sigma_0'',\tau_0'')$ both have $(\Sigma_0''',\tau_0''')$ as a common descendant.  This is a contradiction.  Therefore, $F' \cap F'' \ne \emptyset$.
\newline
\newline
Secondly, suppose that all connected components of $F' \cap F''$ are vertical intervals. Let $N$ be a regular neighborhood of $F' \cap F''$ with respect to some sufficiently fine triangulation (fine enough that it does not intersect $\tau$). Consider $S=\text{cl}(\partial N \backslash (\Sigma_0 \times \partial I))$. Then $S$ is a surface whose connected components are vertically proper embedded surfaces which are homeomorphic to $S^1 \times I$ or $I \times I$. Therefore, the components of $S$ are destabilization surfaces of $(\Sigma_0,\tau_0)$. Let $(\Sigma_0''',\tau_0''')$ denote the result of destabilizing along every component of $S$. It follows that $(\Sigma_0''',\tau_0''')$ is a descendant of both $(\Sigma_0',\tau_0')$ and $(\Sigma_0'',\tau_0'')$. Thus, if all connected components of $F' \cap F''$ are vertical intervals, then we have a contradiction.
\newline
\newline
Therefore it sufficient to prove that $F'\cap F''$ contains only vertical intervals. Suppose then that $F' \cap F''$ has connected components which are not vertical intervals. Suppose furthermore that $F' \cap F''$ has the minimum number of connected components of all pairs of vertically proper embedded destabilization surfaces $G',G''$ for which the destabilizations have no common descendant and for which $G' \cap G''$ has connected components which are not vertical intervals. There are three cases to consider:

\begin{enumerate}[(A)]
\item $F'$ and $F''$ are vertically proper embedded annuli.
\item $F'$ is a vertically proper embedded disc and $F''$ is a vertically proper embedded annulus (or vice versa).
\item $F'$ and $F''$ are both vertically proper embedded discs.
\end{enumerate} 

\noindent Case (A) follows exactly as in \cite{kuperberg}. In particular, it is shown that we may assume that the connected components of $F' \cap F''$ are all vertical intervals. An argument similar to that above then shows that this yields a contradiction.  We thus need only consider cases (B) and (C).
\newline
\newline
\underline{Case (B): $F' \approx I \times I$ and $F'' \approx S^1 \times I$.} Each connected component of $F' \cap F''$ is either
\begin{enumerate}[(i)]
\item an embedded circle,
\item a partisan interval, or 
\item a vertical interval.
\end{enumerate}
Note that since $F''\cap ((\partial \Sigma_0) \times I)=\emptyset$, $F'\cap F''$ cannot have any connected components which are horizontal intervals or cornered intervals.
\newline
\newline
\underline{Sub-case (B.i):} Let $c$ be an embedded circle in $F' \cap F''$. If $c$ bounds a disc on both $F'$ and $F''$, then we may suppose that $c$ is innermost in the sense that the discs bounded by $c$ contain no other connected components of $F' \cap F''$. Apply Lemma \ref{two_inessential} with $R_1=F'$ and $R_2=F''$ to get a vertically proper destabilization surface $R_1'$ which is ambient isotopic to $F'$ and such that $R_1' \cap F''$ has fewer connected components than $F' \cap F''$. This contradicts the hypotheses on $F'$ and $F''$, so $c$ cannot be inessential on both $F'$ and $F''$.
\newline
\newline
Since $c$ necessarily bounds a disc in $F'$, $c$ must be horizontal in $F''$. We may suppose that $c$ is chosen among all components of $F' \cap F''$ so that it is innermost: the disc which it bounds in $F'$ contains no other components of $F'\cap F''$. By Lemma \ref{cutsoffthreeball}, it follows that $F''$ cuts off a $3$-ball from $\Sigma_0 \times I$. Since there are necessarily many such inessential destabilizations of $(\Sigma_0,\tau_0)$ whose defining surfaces do not intersect $F'$, this contradicts our choice of $F'$ and $F''$. We can conclude that $F' \cap F''$ contains no components which are embedded circles.
\newline
\newline
\underline{Sub-case (B.ii):} By sub-case (B.i), we may assume that $F' \cap F''$ contains no components which are embedded circles. Let $e$ be a partisan interval in $F' \cap F''$. Since $F''$ is a vertically proper embedded annulus, it follows that there is a $t\in \{0,1\}$ such that both endpoints of $e$ are in $\Sigma_0 \times \{t\}\subset \Sigma_0 \times I$. Since $F'\cap F''$ contains no embedded circles, we may choose $e$ among all partisan intervals in $F' \cap F''$ such that $e$ is innermost in the sense of Lemma \ref{partisan_top}. Let $R_1=F'$ and $R_2=F''$. We apply the construction of Lemma \ref{partisan_top} to $R_1,R_2$ to obtain a vertically proper embedded surface $R_1'$ which is ambient isotopic to $R_1=F'$ and such that $R_1' \cap F''$ has fewer connected components than $F' \cap F''$. Since this contradicts the choice of $F'$ and $F''$, it follows that $F'\cap F''$ has no connected components which are partisan intervals.
\newline
\newline
Therefore, we have shown that $F' \cap F''$ is a vertical bunch.  As previously shown, $(\Sigma_0',\tau_0')$ and $(\Sigma_0'',\tau_0'')$ have a common descendant. As this is a contradiction, claims (2)-(4) follow for case (B).
\newline
\newline
\underline{Case (C): $F' \approx I \times I$ and $F''\approx I \times I$:} Each connected component of $F' \cap F''$ is either:
\begin{enumerate}[(i)]
\item an embedded circle,
\item a partisan interval,
\item a cornered interval,
\item a horizontal interval, or
\item a vertical interval.
\end{enumerate}
Each case is considered in the stated order. Each sub-case ($j$) relies on the non-existence of connected components occurring in sub-case ($k$), where $k<j$.
\newline
\newline
\underline{Sub-case (C.i):} Let $c$ be an embedded circle in $F' \cap F''$. Then $c$ is an inessential circle in $F'$ and $F''$.  We suppose that $c$ is is innermost in the sense of Lemma \ref{two_inessential}. Apply Lemma \ref{two_inessential} with $R_1=F'$ and $R_2=F''$ to obtain a vertically proper embedded destabilization disc $R_1'$ ambient isotopic to $F'$ such that $R_1'\cap F''$ has fewer connected components than $F'\cap F''$. As this contradicts the hypotheses on $F',F''$, it follows that $F' \cap F''$ has no connected components which are embedded circles.
\newline
\newline
\underline{Sub-case (C.ii):} By sub-case (C.i), $F' \cap F''$ has no connected components which are embedded circles.  Let $e$ be a partisan interval in $F' \cap F''$. 
\newline
\newline
Suppose that there is a $t\in\{0,1\}$ such that both endpoints of $e$ are in $\Sigma_0 \times \{t\} \subset \Sigma_0 \times I$. In this case, we let $R_1=F'$ and $R_2=F''$.  We apply Lemma \ref{partisan_top} to $R_1,R_2$ to obtain a surface $R_1'$ which is ambient isotopic to $F'$ and such that $R_1' \cap F''$ has fewer connected components than $F' \cap F''$. This contradicts the hypotheses on $F'$ and $F''$.  Hence, it follows that both of the endpoints of $e$ are not in $\Sigma_0 \times \{t\}$ for $t\in\{0,1\}$.
\newline
\newline
Since $e$ is partisan, it follows that there is a connected component $C$ of $(\partial \Sigma_0) \times I$ such that both endpoints of $e$ are in $C$. Then $F' \cap C$ and $F'' \cap C$ each contain at least one connected component. Suppose that $F' \cap C$ and $F'' \cap C$ both contain exactly one component. By Lemma \ref{lemma_disc_genus_red}, it follows that $F'$ and $F''$ are both not genus reducing. If both are num reducing, then $\Sigma_0'$ and $\Sigma_0''$ are homeomorphic. Moreover, there is a homeomorphism $h:\Sigma_0' \to \Sigma_0''$ such that $h \times \text{id}$ takes $\tau_0'$ to $\tau_0''$. Thus, at most one of $F'$ and $F''$ is num reducing. Suppose $F'$ is num reducing and $F''$ is not num reducing. Then there is a connected component $C_2$ of $(\partial \Sigma_0) \times I$ such that $F'' \cap C_2$ contains exactly one component and $F' \cap C_2=\emptyset$. Then the destabilization along $F''$ can be considered as an inessential stabilization along $C_2$ (i.e. adding a 3-ball along $C_2$). Since $F'$ and $C_2$ are disjoint, stabilizing along $C_2$ and destabilizing along $F'$ may be done in any order. It follows that $(\Sigma_0',\tau_0')$ and $(\Sigma_0'',\tau'')$ have a common descendant. This is a contradiction.  
\newline
\newline
Suppose then that neither $F'$ nor $F''$ is num reducing. Then these destabilizations can be considered as stabilizations. If there is a connected component $C_3 \ne C$ of $(\partial \Sigma_0) \times I$, such that $F'\cap C_3 \ne \emptyset$ and $F'' \cap C_3 \ne \emptyset$, then both destabilizations may be replaced with inessential stabilizations along $C_3$. If there are distinct connected components $C_4,C_5$ of $(\partial \Sigma_0)\times I$ such that $F'\cap C_4 \ne \emptyset$ and $F'' \cap C_5 \ne \emptyset$, then the destabilizations may be replaced with inessential stabilizations along $C_4$ and $C_5$, respectively. It follows that $(\Sigma_0',\tau_0')$ and $(\Sigma_0'',\tau'')$ have a common descendant. This is a contradiction.
\newline
\newline
Thus, at least one of $F'\cap C$ and $F'' \cap C$ must contain two connected components. Say $F'$ satisfies this property. Then let $R_1=F'$ and $R_2=F''$. We apply Lemma \ref{partisan_side} to $R_1$ and $R_2$ and an innermost partisan interval $e$ (that this is possible follows from sub-case (C.i)). Then there is a surface $R_2'$ which is ambient isotopic to $F''$ in $\Sigma_0 \times I$ and such that $F' \cap R_2'$ has fewer connected components than $F' \cap F''$. This contradicts the hypotheses on $F'$ and $F''$. Thus, $F' \cap F''$ has no connected components which are partisan intervals.  
\newline
\newline
\underline{Sub-case (C.iii):} By sub-cases (C.i) and (C.ii), we have that $F' \cap F''$ has no connected components which are embedded circles or partisan intervals. Let $e$ be a cornered interval.  Let $R_1=F'$ and $R_2=F''$. We apply Lemma \ref{corner_lemma} to $R_1$, $R_2$, and an innermost cornered interval $e$. Then there is a surface $R_1'$ which is ambient isotopic to $F'$ and having the property that $R_1' \cap F''$ has fewer connected components than $F' \cap F''$. This contradicts the hypotheses on $F'$ and $F''$.  Thus, $F' \cap F''$ contains no connected components which are cornered intervals.
\newline
\newline
\underline{Sub-case (C.iv):} By sub-cases (C.i), (C.ii), and (C.iii), we have that $F' \cap F''$ contains no connected components which are embedded circles, partisan intervals, or cornered intervals. If $F'\cap F''$ contains a horizontal interval, then it cannot contain any vertical intervals.  Thus, we assume in this case that $F' \cap F''$ is a nonempty horizontal bunch.
\newline
\newline
Either (a) there is exactly one component $C$ of $(\partial \Sigma_0) \times I$ such that $F' \cap C \ne \emptyset$ and $F'' \cap C \ne \emptyset$ or (b) there are exactly two components $C_1,C_2$ such that $F' \cap C_i \ne \emptyset$ and $F'' \cap C_i \ne \emptyset$ for $i=1,2$. 
\newline
\newline
Suppose that (a) holds. Set $R_1=F'$ and $R_2=F''$.  Then $R_1$ satisfies the hypotheses of Lemma \ref{horiz_bunches}. In any of the possible outcomes of the Lemma, we obtain by ``compression'' a surface $R_1'$ which is ambient isotopic to $R_1$ and which has the property that $R_1'\cap F''$ has fewer connected components than $F' \cap F''$. This contradicts the hypotheses on $F'$ and $F''$. 
\newline
\newline
Therefore, if $F'\cap F''$ is a horizontal bunch, then case (b) holds. Then either both $F'$ and $F''$ are num reducing or neither are num reducing. If neither are num reducing, we may replace the destabilization along $F$ and $F'$ with an inessential stabilization along any $C_i$ which is not distinguished.  This contradicts the fact that $(\Sigma_0',\tau_0')$ and $(\Sigma_0'',\tau_0'')$ do not have a common descendant. Suppose then that both $F'$ and $F''$ are num reducing. Then $\Sigma_0'$ and $\Sigma_0''$ are homeomorphic as they have the same genus and the same number of boundary components. Moreover, there is a homeomorphism $h:\Sigma_0' \to \Sigma_0''$ such that $h \times \text{id}$ takes $\tau_0'$ to $\tau_0''$. This contradicts the hypothesis that $(\Sigma_0',\tau_0')$ and $(\Sigma_0'',\tau_0'')$ have no common descendant. Thus, $F' \cap F''$ cannot be a horizontal bunch.
\newline
\newline
It follows then that $F'\cap F''$ must be a vertical bunch.  As previously shown, $(\Sigma_0',\tau_0')$ and $(\Sigma_0'',\tau_0'')$ have a common descendant. This is a contradiction. 
\end{proof}

\begin{remark} The proof of Lemma \ref{main_1} shows that two irreducible representatives have the same number of distinguished boundary components.
\end{remark}

\subsection{Lemma \ref{main_2}: Stability of Decompositions} \label{sec_stab_decomp} In this section, we prove that some decompositions are preserved by destabilizations. In particular, we prove that if a concatenation is destabilized along a vertically proper embedded surface, then the result is also a concatenation of long knots in thickened surfaces from the same stability classes. This is the content of Lemma \ref{main_2}. First we need the following supporting lemmas.

\begin{figure}
\[
\begin{array}{cc}
\begin{array}{c} \scalebox{1}{\psfig{figure=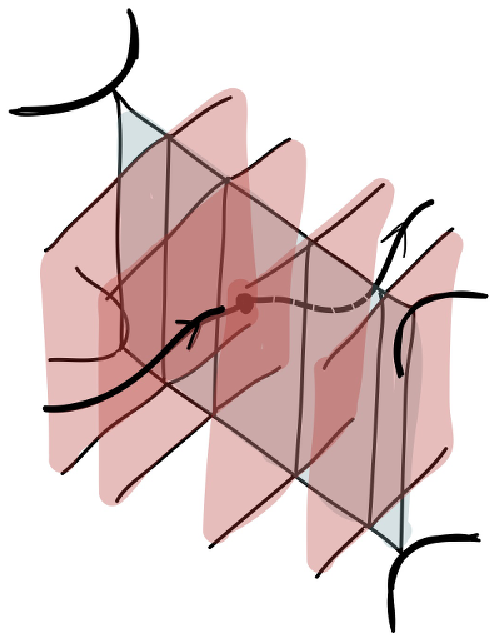}}\end{array} & \begin{array}{c} \scalebox{1}{\psfig{figure=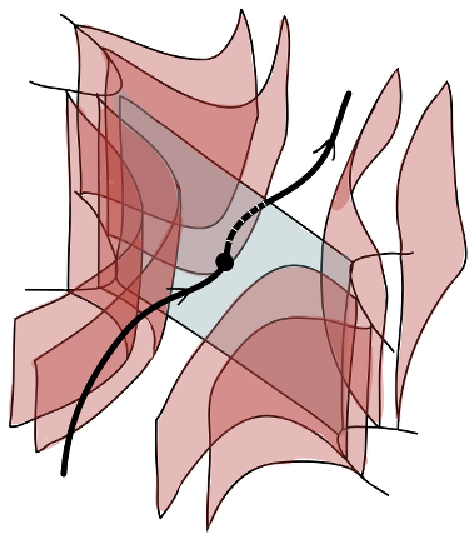}}\end{array}
\end{array}
\]
\caption{$R$ (blue) and $W$ (red) intersecting in a vertical bunch (left hand side of figure) and the sub-rectangles of $W$ from the proof of the Exploding Lemma.}\label{explode_fig}
\end{figure}

\begin{lemma}[The Exploding Lemma]\label{explode} Let $(\Sigma,\tau)=(\Sigma_1\#\Sigma_2,\tau_1\#\tau_2)$ and let $R$ be the vertically proper embedded disc in $\Sigma \times I$ defining the concatenation. Let $W$ be a destabilization surface of $(\Sigma,\tau)$ such that $W \cap R$ is a nonempty vertical bunch of $R$. Cutting $R$ along the vertical bunch $W \cap R$ gives a set of sub-rectangles of $R$: $\{R_1,\ldots,R_n\}$. Let $R_*$ be the sub-rectangle which intersects $\tau$. Let $(\Sigma',\tau')$ be the long knot obtained from $(\Sigma,\tau)$ by first cutting along $W$ and then cutting along each of the vertically proper embedded discs $\{R_1,\ldots,R_n\}\backslash \{R_*\}$. Then the following hold.
\begin{enumerate}
\item $(\Sigma',\tau')=(\Sigma_1'\#\Sigma_2',\tau_1'\#\tau_2')$, with $R_*$ the surface defining the concatenation, and for $i=1,2$, $(\Sigma_i',\tau_i')\sim_s(\Sigma_i,\tau_i)$, and 
\item if $W$ is genus reducing, then $g(\Sigma')<g(\Sigma)$.
\end{enumerate}
\end{lemma}
\begin{proof} Note that after cutting $\Sigma \times I$ along $W$, we obtain a long knot $(\Phi,\gamma)$ stably equivalent to $(\Sigma,\tau)$.  Each of the discs in $\{R_1,\ldots,R_n\}\backslash \{R_*\}$ corresponds to a vertically proper embedded destabilization surface of $(\Phi,\gamma)$ of $\Phi \times I$. Thus, it follows that $g(\Sigma')<g(\Sigma)$ when $W$ is genus reducing.  
\newline
\newline
Now, the vertical bunch $W \cap R$ of $R$ also divides $W$ into sub-rectangles (if $W$ is a disc) or sectors (if $W$ is an annulus). Since $R$ is disconnecting, each sub-rectangle (or sector) of $W$ must lie in one of $\Sigma_1 \times I$ or $\Sigma_2 \times I$ (see the right hand side of Figure \ref{explode_fig}). Hence, each sub-rectangle (sector) of $W$ is a vertically proper embedded destabilization surface of either $(\Sigma_1,\tau_1)$ or $(\Sigma_2,\tau_2)$. By pushing off the two-sided surface $R_*$ if necessary, we may assume that no sub-rectangle (sector) intersects $R_*$. Consider the following alternate construction of $(\Sigma',\tau')$. Cut along $R$ and then identify the resulting two copies of $R_*$. This changes neither the stability class of $(\Sigma,\tau)$ nor the stability classes of the left and right parts, $(\Sigma_i,\tau_i)$ for $i=1,2$. Cutting along each of the sub-rectangles (sectors) of $W$ (each of which is a destabilization surface) does not change the stability class of either component $(\Sigma_1,\tau_1)$ or $(\Sigma_2,\tau_2)$. This completes the proof of the lemma.
\end{proof}

\begin{lemma}[The Num Reducing Lemma]\label{num_reduce} Let $(\Sigma,\tau)=(\Sigma_1\#\Sigma_2,\tau_1\#\tau_2)$ and let $R$ be the vertically proper embedded disc in $\Sigma \times I$ defining the concatenation. Let $W$ be a vertically proper embedded destabilization disc of $(\Sigma,\tau)$ such that $W \cap R$ is a nonempty vertical bunch of $R$ and $W$. Cutting $W$ along the vertical bunch $W \cap R$ gives a set of sub-rectangles of $W$, $\mathscr{W}=\{W_1,\ldots,W_m\}$. Cutting $R$ along the vertical bunch $W \cap R$ gives a set of sub-rectangles of $R$, $\{R_1,\ldots,R_n\}$. Let $R_*$ be the sub-rectangle which intersects $\tau$. Let $(\Sigma',\tau')$ be the long knot obtained by cutting $\Sigma \times I$ along $R$ and subsequently identifying the two copies of $R_*$. Then the following hold:
\begin{enumerate}
\item $R_*$ defines a concatenation $(\Sigma',\tau')=(\Sigma_1'\# \Sigma_2',\tau_1'\# \tau_2')$ such that $(\Sigma_i',\tau_i')\sim_s (\Sigma_i,\tau_i)$ for $i=1,2$, and
\item if $W$ is num reducing, there is a non-empty set of pairwise disjoint destabilization surfaces $\mathscr{X} \subset \mathscr{W}$ of $(\Sigma',\tau')$ such that $X \cap R_*=\emptyset$ for all $X \in \mathscr{X}$, and such that destabilizing $\Sigma' \times I$ along all the $X \in \mathscr{X}$ gives a long knot $(\Sigma'',\tau'')$ such that $(\Sigma'',\tau'')=(\Sigma_1''\#\Sigma_2'',\tau_1''\#\tau_2'')$, $(\Sigma_i'',\tau_i'')\sim_s (\Sigma_i,\tau_i)$ for $i=1,2$, and $\emph{num}(\Sigma'')<\emph{num}(\Sigma)$.
\end{enumerate}
\end{lemma}
\begin{proof} The first claim follows exactly as in the proof of Lemma \ref{explode}. For the second claim, we define a graph $G$ as follows.  Let $U$ be a component of $(\partial\Sigma) \times I$. Define $v_U$ to be a vertex of $G$ if $U$ intersects some sub-rectangle in $\mathscr{W}$. Two vertices $v_{U_1}$, $v_{U_2}$ (not necessarily distinct) are connected by an edge if and only if there is a sub-rectangle whose opposite sides intersect $U_1$ and $U_2$. In particular, there is a one-to-one correspondence between the edges of $G$ and the elements of $\mathscr{W}$. Note that $G$ is connected.  Let $U_0$ be the distinguished boundary component containing $\tau(0)$ and $U_{1}$ the distinguished boundary component containing $\tau(1)$. By hypothesis, $U_0 \ne U_1$. We choose a path from $v_{U_0}$ to $v_{U_1}$ in $G$. Let $\mathscr{X}$ be the subset of $\mathscr{W}$ corresponding to the edges of the path. It is easy to see that cutting along the sub-rectangles in $\mathscr{X}$ reduces the number of distinguished boundary components from $2$ to $1$.
\end{proof}

\begin{proof}[Proof of Lemma \ref{main_2}] Let $R$ be the vertically proper embedded disc $I \times I$ defining the given concatenation $(\Sigma_1 \# \Sigma_2,\tau_1 \# \tau_2)$. By Lemma \ref{lemma_disc_genus_red}, there is exactly one connected component $C$ of $(\partial \Sigma) \times I$ for which $R \cap C \ne \emptyset$. Moreover, $R \cap C$ consists of two disjoint intervals $I_1$ and $I_2$.  Also note that $\tau$ and $R$ intersect transversally and that $|\text{im}(\tau) \cap R|=1$.
\newline
\newline
The proof of the theorem has several key steps. The key steps are:
\begin{enumerate}[(I)]
\item It is shown that when $W$ is a vertically proper embedded disc or annulus, that we may always assume that $W \cap R$ is a (possibly empty) vertical bunch.
\item If $(\Sigma,\tau)$ is not of minimal genus, we take any vertically proper embedded surface which is genus reducing and apply Lemma \ref{explode}. This step is repeated until we have a long knot $(\Sigma',\tau')$ which is of minimal genus.
\item If $(\Sigma',\tau')$ is of minimal genus but does not have the minimal number of distinguished boundary components, we take a vertically proper embedded destabilization disc which is num reducing. We then apply Lemma \ref{num_reduce} to obtain a long knot $(\Sigma_0,\tau_0)$ which has both minimal genus and minimal number of distinguished boundary components among those representatives of minimal genus. 
\end{enumerate}
By Lemma \ref{main_2}, such a sequence exists. By Lemmas \ref{explode} and \ref{num_reduce}, we have that $(\Sigma_0,\tau_0)=(\Sigma_1'\#\Sigma_2',\tau_1'\#\tau_2')$ where $(\Sigma_i',\tau_i')\sim_s(\Sigma_i,\tau_i)$ for $i=1,2$. Then this would complete the proof of the lemma. Since the validity of Steps (II) and (III) follow (via Lemmas \ref{explode} and \ref{num_reduce}) from the validity of Step (I), it suffices to prove the validity of Step (I).
\newline
\newline
\underline{Step (I):} Suppose that $W$ is a vertically proper embedded destabilization surface (which is either a disc or an annulus).  By a general position argument, we may suppose that $W$ and $R$ are chosen among vertically proper embedded destabilization surfaces $W'$ and decomposition surfaces $R'$, ambient isotopic to $W$ and $R$, respectively, so that $W \cap R$ has the fewest number of connected components. Note that $R'$ preserves the decomposition in the sense that it defines a concatenation whose left and right parts are stably equivalent to the left and right parts of the concatenation of $R$. Also, $W'$ is genus reducing, num reducing, or inessential, whenever $W$ is genus reducing, num reducing, or inessential, respectively.
\newline
\newline
If $W \cap R=\emptyset$, then $W$ and $R$ satisfy the hypotheses of (I). Suppose then that $W \cap R \ne \emptyset$.
\newline
\newline
There are two cases to consider, depending on whether $W$ is an annulus or a disc.
\begin{enumerate}[(A)]
\item $W$ is a vertically proper embedded annulus.
\item $W$ is a vertically proper embedded disc.
\end{enumerate}

\noindent Each case is considered in turn.
\newline
\newline
\underline{Case (A): $W$ is an annulus:} Since $W$ is an annulus, destabilizing along $W$ does not change the number of distinguished boundary components of $(\Sigma,\tau)$.  Now, each component of $W \cap R$ is either:
\begin{enumerate}[(i)]
\item an embedded circle,
\item a partisan interval on $R$, or
\item a vertical interval on $R$.
\end{enumerate}
Note that since $W$ is an annulus, there are no intersections in cornered intervals or horizontal intervals.
\newline
\newline
\underline{Sub-case (A.i):} Suppose that $W \cap R$ contains an embedded circle $c$. Since $R$ is a disc, it follows that $c$ must bound a disc $D_1$ in $R$.  Choose $c$ so that it is innermost in the sense that $D_1$ contains no other components of $W \cap R$. There are two cases to consider: (1) $c$ is inessential in $W$ and (2) $c$ is a horizontal in $W$.
\newline
\newline
Suppose first that $c$ is inessential in $W$, so that $c$ bounds a disc $D_2$ in $W$ and is innermost in $W$. Then by Lemma \ref{two_inessential}, $D_1 \cup D_2$ is a $2$-sphere bounding a $3$-ball $B_{12}$. If $\text{im}(\tau) \cap D_1 \ne \emptyset$, then there must be at least one other point in $D_1 \cup D_2$ which intersects $\tau$.  However, since $W \cap \text{im}(\tau)=\emptyset$, this is impossible. Thus, $(D_1 \cup D_2) \cap \text{im}(\tau)=\emptyset$. Let $R_1=R$ and $R_2=W$ and apply Lemma \ref{two_inessential}. There is a vertically proper embedded annulus $W'$ which is ambient isotopic to $W$ and for which $W'\cap R$ has fewer connected components than $W \cap R$. As this is a contradiction, we conclude that $W \cap R$ can have no components which are inessential circles in both $W$ and $R$.
\newline
\newline
Suppose then that $c$ is horizontal in $W$. By Lemma \ref{cutsoffthreeball}, $W$ cuts off a $3$-ball $B_{12}$ from $\Sigma \times I$. It follows that $B_{12} \cap \text{im}(\tau)=\emptyset$. Then we may replace $W$ with any vertically proper embedded annulus in $\Sigma \times I$ which cuts off a $3$-ball. Since such an annulus may be chosen so that it has no intersections with $R$, we have a contradiction of the assumptions on $W$ and $R$. Thus, we may assume that $W \cap R$ contains no components which are embedded circles.
\newline
\newline
\underline{Sub-case (A.ii):} By sub-case (A.i), we may assume that $W \cap R$ contains no components which are embedded circles. Let $e$ be a partisan interval of $W \cap R$ in $R$. Since $W$ is vertically proper embedded, there is a $t \in \{0,1\}$ such that $e(0)$ and $e(1)$ are both in $\Sigma \times \{t\} \subset \Sigma \times I$. Let $R_1=R$ and $R_2=W$ and apply Lemma \ref{partisan_top}. Let $D_1,D_2,D_3$ be as in Lemma \ref{partisan_top}. We may suppose that $e$ is chosen so that it is innermost in the sense that $D_1 \cup D_2$ contains no connected components of $R \cap W$ other than $e$.  Since $\tau(0)$ and $\tau(1)$ are in $(\partial \Sigma) \times I$, it follows that $D_3 \cap \text{im}(\tau)=\emptyset$. Hence, $B_{12} \cap \text{im}(\tau)=\emptyset$, for otherwise $D_1 \cup D_2 \subset R \cup W$ would have more than one intersection with $\tau$. By Lemma \ref{partisan_top}, there is a vertically proper embedded destabilization surface $W'$ which is ambient isotopic to $W$ and for which $R \cap W'$ has fewer connected components than $R \cap W$. As this contradicts our choice of $W$ and $R$, we may conclude that $R \cap W$ contains no components which are partisan intervals.
\newline
\newline
\underline{Sub-case (A.iii):} By sub-cases (A.i) and (A.ii), we may assume that $W \cap R$ contains no embedded circles or partisan intervals. It follows that $W \cap R$ is a non-empty vertical bunch. Therefore, when $W$ is an annulus, we may assume that $W \cap R$ is a vertical bunch. This completes Step (I.A).
\newline
\newline
\underline{Case (B): W is a disc:} By hypothesis, $R$ is a disc $I \times I$.  Then each connected component of $R \cap W$ is either:
\begin{enumerate}[(i)]
\item an embedded circle, 
\item a partisan interval,
\item a cornered interval, 
\item a horizontal interval, or
\item a vertical interval.
\end{enumerate}  
We argue in the stated order that such intersections violate the hypotheses on $R$, $W$, $(\Sigma,\tau)$, or $(\Sigma_i,\tau_i)$ for $i=1,2$.
\newline
\newline
\underline{Sub-case (B.i):} Let $c$ be a connected component of $R \cap W$ which is an embedded circle. Then $c$ is inessential in both $R$ and $W$. Let $R_1=R$, $R_2=W$ and $D_i \subset R_i$ discs such that $c=\partial D_i$ for $i=1,2$. By Lemma \ref{two_inessential}, $D_1 \cup D_2$ is a 2-sphere which bounds a 3-ball $B_{12}$. There are two possibilities: $B_{12} \cap \text{im}(\tau) \ne \emptyset$ and $B_{12}\cap \text{im}(\tau) = \emptyset$. If $B_{12} \cap \text{im}(\tau) \ne \emptyset$, then we must have that $D_1 \cap \text{im}(\tau) \ne \emptyset$. Since $\text{im}(\tau) \cap (D_1 \cup D_2)$ must contain at least two points, this contradicts the fact that $W \cap \text{im}(\tau)=\emptyset$ and $|R \cap \text{im}(\tau)|=1$. Therefore, we must have that $B_{12} \cap \text{im}(\tau)=\emptyset$.  By Lemma \ref{two_inessential}, there is a vertically proper embedded disc $W'$ which is ambient isotopic to $W$, has no intersections with $\tau$, and such that $W'\cap R$ has fewer connected components than $W \cap R$. As this contradicts the hypotheses on $W$ and $R$, we conclude that $W \cap R$ has no connected components which are embedded circles.
\newline
\newline
\underline{Sub-case (B.ii):} By (B.i), we may assume that $R \cap W$ contains no connected components which are embedded discs.  Let $e$ be a partisan interval in $W \cap R$. There are two cases to consider: (1) there is a $t\in\{0,1\}$ such that both endpoints of $e$ are in $\Sigma \times \{t\} \subset \Sigma \times I$ and (2) there is exactly one component $C$ of $(\partial \Sigma) \times I$ which contains both endpoints of $e$.
\newline
\newline
Suppose that case (1) holds. Let $R_1=R$ and $R_2=W$.  We apply Lemma \ref{partisan_top} to $(\Sigma,\tau)$, $R_1$ and $R_2$. We may suppose that $e$ is innermost in the sense that $D_1 \cup D_2$ contains no other partisan intervals or embedded circles of $R \cap W$. Let $B_{12}$ be as in (2) of Lemma \ref{partisan_top}. Then either $B_{12}\cap \text{im}(\tau)=\emptyset$ or $B_{12}\cap \text{im}(\tau) \ne \emptyset$. If $B_{12}\cap \text{im}(\tau)=\emptyset$, it follows from Lemma \ref{partisan_top} that there is a vertically proper embedded disc $W'$ which is ambient isotopic to $W$, has no intersections with $\tau$, and such that $W' \cap R$ has fewer connected components than $W \cap R$. As this contradicts the choice of $R$ and $W$, it follows that $B_{12} \cap \text{im}(\tau) \ne \emptyset$.  Then $D_1 \cup D_2 \cup D_3$ must intersect $\text{im}(\tau)$ in at least two points. Since $W \cap \text{im}(\tau)=\emptyset$, $|R \cap \text{im}(\tau)|=1$, and $(\Sigma \times \{t\}) \cap \text{im}(\tau)=\emptyset$, this cannot occur. Thus, there can be no partisan intervals in $R \cap W$ as in case (1).
\newline
\newline
Suppose that case (2) holds. Recall that $R$ is disconnecting, so that there is exactly one component $C$ of $(\partial \Sigma) \times I$ such that $R \cap C \ne \emptyset$. Moreover, $R\cap C$ consists of two disjoint intervals. Let $R_1=R$ and $R_2=W$. We apply Lemma \ref{partisan_side} to $(\Sigma,\tau)$, $R_1$, and $R_2$. Let $D_1$, $D_2$, $D_3$ be as in Lemma \ref{partisan_side}. We may assume that $e$ is innermost in the sense that $D_1 \cup D_2$ contains no connected components of $R \cap W$ other than $e$. Let $B_{12}$ be as in (2) of Lemma \ref{partisan_side}. Again, there are two sub-cases to consider: $B_{12}\cap \text{im}(\tau)=\emptyset$ or $B_{12}\cap \text{im}(\tau) \ne \emptyset$. If $B_{12}\cap \text{im}(\tau)=\emptyset$, it follows from Lemma \ref{partisan_side} that there is a vertically proper embedded disc $W'$ which is ambient isotopic to $W$ such that $W' \cap \text{im}(\tau)=\emptyset$ and $W' \cap R$ has fewer connected components than $W \cap R$. As this contradicts the hypotheses on $W \cap R$, we must have that $B_{12}\cap \text{im}(\tau)\ne\emptyset$. Since $W \cap \text{im}(\tau)=\emptyset$, we must have that $D_3 \cap \text{im}(\tau) \ne \emptyset$. Since $R_1$ is disconnecting, we must have that $D_3$ contains only one of $\tau(0)$ and $\tau(1)$. The hypotheses on $R$ imply that $|D_1 \cap \text{im}(\tau)|=1$. Thus, either $\text{im}(\tau_1) \subset B_{12}$ or $\text{im}(\tau_2) \subset B_{12}$.  Using the bead-on-a-string isotopy (Lemma \ref{beadonastring}), we see that either $(\Sigma_1,\tau_1)$ or $(\Sigma_2,\tau_2)$ stabilizes to a long classical knot.  This is a contradiction.  Thus, we conclude that $R \cap W$ has no partisan intervals.
\newline
\newline
\underline{Sub-case (B.iii):} By sub-cases (B.i) and (B.ii), we may assume that $R \cap W$ contains no connected components which are embedded circles or partisan intervals. Suppose that $R \cap W$ contains a connected component which is a cornered interval. Let $R_1=R$ and $R_2=W$. We apply Lemma \ref{corner_lemma} to $(\Sigma,\tau)$, $R_1$, and $R_2$. We may choose a cornered interval $e$ so that $D_1 \cup D_2$ contains no connected components of $R \cap W$ other than $e$. Let $B_{12}$ be the 3-ball in (2) of Lemma \ref{corner_lemma}. Then either $B_{12} \cap \text{im}(\tau)=\emptyset$ or $B_{12} \cap \text{im}(\tau) \ne \emptyset$. If $B_{12} \cap \text{im}(\tau)=\emptyset$, it follows from Lemma \ref{corner_lemma} that there is a vertically proper embedded disc $W'$ which is ambient isotopic to $W$, has no intersections with $\tau$, and has the property that that $W' \cap R$ has fewer connected components than $W \cap R$. As this contradicts our choice of $W$ and $R$, we must have that $B_{12} \cap \text{im}(\tau) \ne \emptyset$. By hypothesis, we must have that $D_2 \cap \text{im}(\tau)=\emptyset$ and $D_3 \cap \text{im}(\tau)=\emptyset$. Since $R$ is disconnecting, $E_{12}$ (see Lemma \ref{corner_lemma}) can contain at most one of $\tau(0)$ and $\tau(1)$. Thus, $D_1 \cap \text{im}(\tau)=R \cap \text{im}(\tau)$. It follows that either $\text{im}(\tau_1) \subset B_{12}$ or $\text{im}(\tau_2) \subset B_{12}$. The bead-on-a-string isotopy (Lemma \ref{beadonastring}) then implies that either $(\Sigma_1,\tau_1)$ or $(\Sigma_2,\tau_2)$ stabilizes to a long classical knot. This contradicts the hypotheses on $(\Sigma_1,\tau_1)$ and $(\Sigma_2,\tau_2)$. Thus we may conclude that $R \cap W$ contains no connected components which are cornered intervals.  
\newline
\newline
\underline{Sub-case (B.iv):} By sub-cases (B.i), (B.ii), and (B.iii), we may assume that $R \cap W$ contains no connected components which are embedded circles, partisan intervals, or cornered intervals. Suppose that $R \cap W$ has a connected component which is a horizontal interval. Hence, $R \cap W$ cannot have a component which is a vertical interval and a component which is a horizontal interval. From these observations, it follows that $R \cap W$ is a non-empty horizontal bunch.  
\newline
\newline
Since $R$ is disconnecting it follows from Lemma \ref{lemma_disc_genus_red} that there is exactly one component $C$ of $(\partial \Sigma) \times I$ such that $R \cap C \ne \emptyset$. Let $R_1=R$ and $R_2=W$. We apply Lemma \ref{horiz_bunches} to $(\Sigma,\tau)$, $R_1$, and $R_2$. Let $r_1^j \subset R_1$ and $r_{2}^k \subset R_2 $ be the rectangles obtained from (1) of Lemma \ref{horiz_bunches}.
\newline
\newline
Suppose first that $j=k=1$ or $j=k=n+1$. Let $B_{12}$ be the 3-ball of Lemma \ref{horiz_bunches} (2.b). There are two sub-cases to consider: $B_{12} \cap \text{im}(\tau)=\emptyset$ and $B_{12} \cap \text{im}(\tau) \ne \emptyset$.  If $B_{12} \cap \text{im}(\tau)=\emptyset$, it follows from Lemma \ref{horiz_bunches} that there is a vertically proper embedded disc $W'$ which is ambient isotopic to $W$, has no intersections with $\tau$, and satisfies the property that $W' \cap R$ has fewer connected components than $W \cap R$. This contradicts the hypotheses on $R$ and $W$. 
\newline
\newline
Suppose then that $B_{12} \cap \text{im}(\tau) \ne \emptyset$. Since $R$ is disconnecting, $D_1$ and $D_2$ may each contain at most one endpoint of $\tau$. Suppose that $D_1$ contains an endpoint of $\tau$ and $D_2$ contains the other endpoint of $\tau$. By hypothesis on $W$, $r_2^k \cap \text{im}(\tau)=\emptyset$. If $r_1^j \cap \text{im}(\tau)=\emptyset$, then $\text{im}(\tau) \subset B_{12}$.  This implies that $(\Sigma,\tau)$ stabilizes to a long classical knot, which contradicts the hypotheses on $\tau$.  Thus, we must have that $r_1^j \cap \text{im}(\tau)=R \cap \text{im}(\tau)$.  However, since $|R \cap \text{im}(\tau)|=1$ and $\tau(0),\tau(1) \in \partial B_{12}$, this contradicts the fact that the unit interval $I$ (which is the domain of $\tau$) is connected. Therefore, either $D_1$ or $D_2$ contains an endpoint of $\tau$ but not both. Without loss of generality, suppose that $D_1$ contains an endpoint of $\tau$. Since the unit interval is connected and $r_2^k \cap \text{im}(\tau)=\emptyset$, we must have that $r_1^j \cap \text{im}(\tau)=R \cap \text{im}(\tau)$. Thus, either $\text{im}(\tau_1) \subset B_{12}$ or $\text{im}(\tau_2) \subset B_{12}$. Applying the bead-on-a-string isotopy (Lemma \ref{beadonastring}), we conclude that either $(\Sigma_1,\tau_1)$ or $(\Sigma_2,\tau_2)$ stabilizes to a long classical knot. This contradicts the hypotheses on $(\Sigma_1,\tau_1)$ and $(\Sigma_2,\tau_2)$. Finally, if neither $D_1$ nor $D_2$ contains an endpoint of $\tau$, then $\partial B_{12} \cap \text{im}(\tau)$ contains either 0 or 1 point, both of which are impossible when $B_{12} \cap \text{im}(\tau) \ne \emptyset$. We have thus proved that for the rectangles $r_1^j$ and $r_2^k$, we cannot have that $j=k=1$ or $j=k=n+1$.
\newline
\newline
The remaining possibility from Lemma \ref{horiz_bunches} is that $1<j,k<n+1$. The proof in this case follows nearly identically to the case that $j,k=1$ or $j,k=n+1$. Hence we leave it as an exercise for the reader.
\newline
\newline
It follows that $R \cap W$ has no connected components which are horizontal intervals.
\newline

\noindent \underline{Sub-case (B.v):} By sub-cases (B.i), (B.ii), and (B.iii), we conclude that $R \cap W$ contains no components which are embedded circles, partisan intervals, or cornered intervals. By sub-case (B.iv), if $W \cap R$ also contains no vertical intervals, then it has no connected components which are horizontal intervals. Thus, the only remaining possibility is that $W \cap R$ is a non-empty vertical bunch.
\newline
\newline
This completes the proof of Step (I).  By the observations recorded at the beginning of the proof, the proof of the lemma is complete. 
\end{proof} 
\subsection{Theorem \ref{main_3}: Non-Commutativity of Non-Classical Linear Primes} \label{sec_non_commute} In this section, we prove that if $A,B,C,D$ are non-classical linear primes such that $A\# B$ is non-classical and $A \#B \leftrightharpoons C \# D$, then $A \leftrightharpoons C$ and $B \leftrightharpoons D$. An immediate consequence is that distinct non-classical linearly prime long virtual knots do not commute in the monoid of long virtual knots. The proof makes use of Lemmas \ref{prime_corr}, \ref{main_1}, and \ref{main_2}.

\begin{proof}[Proof of Theorem \ref{main_3}] Let $K_1,K_2,K_3,K_4$ be non-classical linearly prime long virtual knots such that $K_1 \# K_2$ is non-classical. Suppose that $\mathscr{F}(K_i)\sim_s(\Sigma_i,\tau_i)$ for $i=1,2,3,4$. Let $(\Sigma_{12},\tau_{12})=(\Sigma_1\#\Sigma_2,\tau_1\#\tau_2)$ and $(\Sigma_{34},\tau_{34})=(\Sigma_{3} \# \Sigma_4,\tau_3\#\tau_4)$. Then $(\Sigma_{12},\tau_{12})\sim_s (\Sigma_{34},\tau_{34})$, $\mathscr{F}^{-1}(\Sigma_{12},\tau_{12})\leftrightharpoons K_1\# K_2$, and $\mathscr{F}^{-1}(\Sigma_{34},\tau_{34})\leftrightharpoons K_3 \# K_4$. By Lemma \ref{main_2}, there is an irreducible descendant $(\Sigma_0,\tau_0)$ of $(\Sigma_{12},\tau_{12})$ such that $(\Sigma_0,\tau_0)=(\Sigma_1'\#\Sigma_2',\tau_1'\#\tau_2')$ where $(\Sigma_i',\tau_i')\sim_s (\Sigma_i,\tau_i)$ for $i=1,2$. By Lemma \ref{main_1}, we have that $(\Sigma_0,\tau_0)$ is also an irreducible descendant of $(\Sigma_{34},\tau_{34})$. Applying Lemma \ref{main_2} to $(\Sigma_{34},\tau_{34})$, we have that $(\Sigma_0,\tau_0)=(\Sigma_3'\#\Sigma_4',\tau_3'\#\tau_4')$ where $(\Sigma_i',\tau_i') \sim_s (\Sigma_i,\tau_i)$ for $i=3,4$. Let $R_{12}$ be the vertically proper embedded disc defining the concatenation $(\Sigma_0,\tau_0)=(\Sigma_1'\#\Sigma_2',\tau_1'\#\tau_2')$. Let $R_{34}$ be the vertically proper embedded disc defining the concatenation $(\Sigma_0,\tau_0)=(\Sigma_3'\#\Sigma_4',\tau_3'\#\tau_4')$.
\newline
\newline
By a general position argument, the discs $R_{12}$ and $R_{34}$ may be chosen so that they intersect transversely. In addition, we suppose that $R_{12}$ and $R_{34}$ are chosen among all decomposition discs $R_{12}'$ and $R_{34}'$, ambient isotopic to $R_{12}$ and $R_{34}$, respectively, which satisfy the property that the number of connected components of $R_{12}' \cap R_{34}'$ is at least the number of connected components of $R_{12} \cap R_{34}$. We note that $R_{12}'$ ($R_{34}'$) preserves the decomposition of $R_{12}$ (resp. $R_{34}$) in the sense that it defines a concatenation whose left and right parts are stably equivalent to the left and right parts of $R_{12}$ (resp. $R_{34}$). Note that we may also assume that $\text{im}(\tau_0)\cap R_{12} \cap R_{34}=\emptyset$.
\newline
\newline
Suppose first that $R_{12} \cap R_{34}=\emptyset$. Then either $R_{12} \subset \Sigma_3' \times I$ or $R_{12} \subset \Sigma_4' \times I$. Suppose that $R_{12} \subset \Sigma_3' \times I$. Since $R_{12}$ is vertically proper embedded, disconnecting, and $|R_{12} \cap \text{im}(\tau_0)|=1$, we have that $R_{12}$ defines a concatenation $(\Sigma_3, \tau_3)=(\Gamma_1 \# \Gamma_2,\gamma_1\#\gamma_2)$. Since $K_3$ is linearly prime, it follows that either $(\Gamma_1,\gamma_1)$ or $(\Gamma_2,\gamma_2)$ stabilizes to the long trivial knot. Since $\tau_0'(0)=\tau_1'(0)=\tau_3'(0)$, and $\tau_1$ does not stabilize to the long trivial knot, we must have that $(\Gamma_1,\gamma_1)=(\Sigma_1',\tau_1')$ and that $(\Gamma_2,\gamma_2)$ stabilizes to a long trivial knot $G_1$.  Then we have that $K_3 \leftrightharpoons K_1 \# G_1 \leftrightharpoons K_1$.  Now, since $R_{12} \subset \Sigma_3' \times I$, it follows that $R_{34} \subset \Sigma_2' \times I$. A similar argument shows that $K_2 \leftrightharpoons K_4$. If $R_{12} \subset \Sigma_4' \times I$, it similarly follows that $K_1 \leftrightharpoons K_2$ and $K_3 \leftrightharpoons K_4$. Thus, the theorem holds if $R_{12} \cap R_{34}=\emptyset$.
\newline
\newline
Now suppose that $R_{12} \cap R_{34}$ is a non-empty vertical bunch. The components of $R_{12} \cap R_{34}$ separate $R_{12}$ and $R_{34}$ into sub-rectangles $\{R_{12}^1,\ldots,R_{12}^m\}$ and $\{R_{34}^1,\ldots,R_{34}^m\}$. Let $R_{12}^*$ be the sub-rectangle of $R_{12}$ that contains $R_{12} \cap \text{im}(\tau_0)$ and $R_{34}^*$ be the sub-rectangle of $R_{34}$ that contains $R_{34} \cap \text{im}(\tau_0)$. Now, cut along $R_{12}$ and identify the two copies of $R_{12}^*$. This gives a long knot $(\Sigma_0',\tau_0')=(\Sigma_1''\#\Sigma_2'',\tau_1''\#\tau_2'')$ where $(\Sigma_0',\tau_0')\sim_s (\Sigma_0,\tau_0)$ and $(\Sigma_i'',\tau_i'') \sim_s (\Sigma_i,\tau_i)$ for $i=1,2$. For each of the sub-rectangles $R_{34}^j$, there is a corresponding vertically proper embedded disc in $\Sigma_0' \times I$. We will also denote each corresponding disc by $R_{34}^j$. Each sub-rectangle may be considered as a vertically proper embedded disc in either $\Sigma_1'' \times I$ or $\Sigma_2'' \times I$. We destabilize along each $R_{34}^j$ to obtain a long knot $(\Sigma_0'',\tau_0'')\sim_s (\Sigma_0',\tau_0')$. Then either $R_{34}^* \subset \Sigma_1' \times I$ or $R_{34}^* \subset \Sigma_2' \times I$. If $R_{12}^* \cap R_{34}* =\emptyset$, then the result follows as in the case that $R_{12} \cap R_{34}=\emptyset$. If $R_{12}^* \cap R_{34}^* \ne \emptyset$, then some side of $R_{12}^*$ coincides with a side of $R_{34}^*$. In this case, we may push $R_{12}^*$ slightly off of $R_{34}^*$ to obtain the case that $R_{12}^* \cap R_{34}^*=\emptyset$.
\newline
\newline
It remains to prove that our hypotheses on $R_{12}$ and $R_{34}$ imply that that every component of $R_{12} \cap R_{34}$ is a vertical interval. We consider each of the possible types of other intersections and show they lead to contradictions. The remaining possibilities are:
\begin{enumerate}[(i)]
\item an embedded circle,
\item a partisan interval,
\item a cornered interval, or
\item a horizontal interval.
\end{enumerate}
We consider each possibility in the stated order.
\newline
\newline
\underline{Case (i):} Let $c$ be a component of $R_{12} \cap R_{34}$ which is an embedded circle.  Then $c$ is inessential in both $R_{12}$ and $R_{34}$. Then $c$ bounds a disc $D_1$ on $R_{12}$ and $D_2$ on $R_{34}$. We suppose that $c$ is innermost in the sense that $D_1 \cup D_2$ contains no components of $R_{12} \cap R_{34}$ other than $c$. We apply Lemma \ref{two_inessential} to $R_1=R_{12}$, $R_2=R_{34}$, $D_1$ and $D_2$. Let $B_{12}$ be the 3-ball bounded by $D_1 \cup D_2$. 
\newline
\newline
Suppose that $B_{12} \cap \text{im}(\tau_0)=\emptyset$. By Lemma \ref{two_inessential}, there is a vertically proper embedded decomposition disc $R_{12}'$ ambient isotopic to $R_{12}$ such that $R_{12}' \cap R_{34}$ has fewer connected components than $R_{12} \cap R_{34}$. As this contradicts the choice of $R_{12}$ and $R_{34}$, we cannot have $B_{12} \cap \text{im}(\tau_0)=\emptyset$.
\newline
\newline
Suppose then that $B_{12} \cap \text{im}(\tau) \ne \emptyset$. Since $B_{12}$ is a 3-ball and $\tau_0$ intersects $R_{12}$ and $R_{34}$ transversely at one point each, we must have that $|D_1 \cap \text{im}(\tau_0)|=1$ and $|D_2 \cap \text{im}(\tau_0)|=1$. Note that $B_{12} \subset \Sigma_i' \times I$ for some $i=1,2$. It follows that $B_{12}$ contains the endpoint $\tau_i'(2-i)$. We apply the bead-on-a-string isotopy (see Lemma \ref{beadonastring}) to move $B_{12}$ towards the endpoint $\tau_i'(2-i)$. This gives a decomposition $(\Sigma_i',\tau_i')=(\Gamma_1 \# \Gamma_2, \gamma_1 \# \gamma_2)$. Since $K_i$ is linearly prime, it must be that either $(\Gamma_1,\gamma_1)$ or $(\Gamma_2,\gamma_2)$ stabilizes to a long trivial knot. Moreover, this isotopy to triviality can be done inside the 3-ball $B_{12}$ so that the arc $\text{im}(\tau_0)\cap B_{12}$ can be taken to be an unknotted arc. Applying Lemma \ref{two_inessential} we get a vertically proper decomposition disc $R_{12}'$ which is ambient isotopic to $R_{12}$, has one intersection with $\tau_0$, and satisfies the property that $R_{12}'\cap R_{34}$ has fewer connected components than $R_{12} \cap R_{34}$. Since this contradicts our choice of $R_{12}$ and $R_{34}$, we conclude that $R_{12} \cap R_{34}$ contains no connected components which are embedded circles. 
\newline
\newline
\underline{Case (ii):} By Case (i), we may assume that $R_{12} \cap R_{34}$ has no connected components which are embedded circles.  Let $e$ be a connected component of $R_{12} \cap R_{34}$ which is a partisan interval. There are two sub-cases to consider:
\begin{enumerate}[(a)]
\item both endpoints of $e$ intersect $\Sigma_0 \times \{t\}$ for some $t \in \{0,1\}$, or
\item both endpoints of $e$ intersect the same connected component of $(\partial \Sigma_0) \times I$.
\end{enumerate}
The possibilities are disjoint, so they may be considered in any order.
\newline
\newline
\underline{Sub-case (ii.a):} Set $R_1=R_{12}$ and $R_2=R_{34}$ and let $D_1 \subset R_{12}$ and $D_2 \subset R_{34}$ be the discs defined in Lemma \ref{partisan_top}. Let $B_{12}$ be the 3-ball given in Lemma \ref{partisan_top} (2). Suppose that $B_{12} \cap \text{im}(\tau_0)=\emptyset$. By Lemma \ref{partisan_top}, there is a vertically proper embedded decomposition disc $R_{12}'$ which is ambient isotopic to $R_{12}$ and such that $R_{12}' \cap R_{34}$ has fewer connected components than $R_{12} \cap R_{34}$. As this contradicts the hypotheses on $R_{12}$ and $R_{34}$, we cannot have that $B_{12} \cap \text{im}(\tau_0)=\emptyset$.
\newline
\newline
Suppose that $B_{12} \cap \text{im}(\tau_0) \ne \emptyset$. Since $\tau_0$ intersects $R_{12}$ and $R_{34}$ transversely and $\tau_0$ does not intersect $\Sigma \times \partial I$, we must have that $|(D_1 \cup D_2 \cup D_3) \cap \text{im}(\tau_0)|=2$. Using an argument similar to Case (i), we can show that this is impossible. We can find a vertically proper embedded decomposition disc $R_{12}'$ ambient isotopic to $R_{12}$ such that $R_{12}' \cap R_{34}$ has fewer connected components than $R_{12} \cap R_{34}$. Thus, we conclude that $R_{12} \cap R_{34}$ can have no connected components which are partisan intervals and satisfy (a).
\newline
\newline
\underline{Sub-case (ii.b):} Since $R_{12}$ is disconnecting, Lemma \ref{lemma_disc_genus_red} implies that there is exactly one component $C$ of $(\partial \Sigma_0) \times I$ such that $R_{12} \cap C \ne \emptyset$. Apply Lemma \ref{partisan_side} with $R_1=R_{12}$ and $R_2=R_{34}$. Let $B_{12}$ be the 3-ball bounded by the 2-sphere $D_1 \cup D_2 \cup D_3$ of Lemma \ref{partisan_side} (2). There are three non-trivial sub-sub-cases to consider:
\begin{enumerate}[(1)]
\item $\text{im}(\tau_0) \cap B_{12}=\emptyset$,
\item $D_3 \cap \text{im}(\tau_0)=\emptyset$ and $|(D_1 \cup D_2) \cap \text{im}(\tau_0)|=2$, and
\item $|(D_1 \cup D_2) \cap \text{im}(\tau_0)|=1$ and $|D_3 \cap \text{im}(\tau_0)|=1$.
\end{enumerate} 
We shall prove that each sub-sub case leads to a contradiction.
\newline
\newline
\underline{Sub-sub-case (ii.b.1):} Suppose that $\text{im}(\tau_0) \cap B_{12}=\emptyset$. By Lemma \ref{partisan_side}, there is a vertically proper embedded decomposition disc $R_{12}'$ which is ambient isotopic to $R_{12}$ and satisfies the property that $R_{12}' \cap R_{34}$ as fewer connected components than $R_{12} \cap R_{34}$. This contradicts the choice of $R_{12}$ and $R_{34}$.
\newline
\newline
\underline{Sub-sub-case (ii.b.2):} Suppose that $D_3 \cap \text{im}(\tau_0)=\emptyset$ and $|(D_1 \cup D_2) \cap \text{im}(\tau_0)|=2$. Then we must have that $D_1 \cap \text{im}(\tau_0)$ and $D_1 \cap \text{im}(\tau_0)$ each contain exactly one point. Using the fact that $K_1$ and $K_2$ are both linearly prime, we may argue as in Case (i) (and Case (ii.a)) that this leads to a contradiction of the choice of $R_{12}$ and $R_{34}$.
\newline
\newline
\underline{Sub-sub-case (ii.b.3):} Suppose that $|(D_1 \cup D_2) \cap \text{im}(\tau_0)|=1$ and $|D_3 \cap \text{im}(\tau_0)|=1$. Suppose that $D_1 \cap \text{im}(\tau_0) \ne \emptyset$. The intersection of $D_3 \subset C$ and $\text{im}(\tau_0)$ must be either $\tau_0(0)$ or $\tau_0(1)$. It follows that either $\text{im}(\tau_1') \subset B_{12}$ or $\text{im}(\tau_2') \subset B_{12}$. Applying the bead-on-a-string isotopy, it follows that either $(\Sigma_1',\tau_1')$ or $(\Sigma_2',\tau_2')$ stabilizes to a long classical knot.  This contradicts the hypotheses on $K_1$ and $K_2$.
\newline
\newline
We conclude that $R_{12} \cap R_{34}$ may contain no connected components which are partisan intervals.
\newline
\newline
\underline{Case (iii):} By Cases (i) and (ii), we may assume that $R_{12} \cap R_{34}$ contains no connected components which are embedded circles or partisan intervals. Let $e$ be a cornered interval of $R_{12} \cap R_{34}$. We apply Lemma \ref{corner_lemma} to $e$ and $R_1=R_{12}$, $R_2=R_{34}$. Let $D_1$, $D_2$ be as given in Lemma \ref{corner_lemma}. Since $R_{12} \cap R_{34}$ contains no embedded circles or partisan intervals, we may suppose that $e$ is chosen to be innermost in the sense that $D_1 \cup D_2$ contains no connected components of $R_{12} \cap R_{34}$ other than $e$. Let $D_3$, $E_{12}$ be as given in Lemma \ref{corner_lemma}. Let $B_{12}$ be the 3-ball as given in Lemma \ref{corner_lemma} (2). It is easy to see that there are exactly three non-trivial sub-cases to consider:
\begin{enumerate}[(a)]
\item $\text{im}(\tau_0) \cap B_{12}=\emptyset$,
\item $|(D_1 \cup D_2) \cap \text{im}(\tau_0)|=2$,
\item $|E_{12} \cap \text{im}(\tau_0)|=1$ and either $|D_1 \cap \text{im}(\tau_0)|=1$ or $|D_2 \cap \text{im}(\tau_0)|=1$ (but not both).
\end{enumerate}
We prove that each sub-case leads to a contradiction.
\newline
\newline
\underline{Sub-case (iii.a):} Suppose that $\text{im}(\tau_0) \cap B_{12}=\emptyset$. By Lemma \ref{corner_lemma}, there is a vertically proper embedded disc $R_{34}'$ which is ambient isotopic to $R_{34}$ and having the property that $R_{12} \cap R_{34}'$ has fewer connected components than $R_{12} \cap R_{34}$. This contradicts the choice of $R_{12}$ and $R_{34}$.
\newline
\newline
\underline{Sub-case (iii.b):} Suppose that $|(D_1 \cup D_2) \cap \text{im}(\tau_0)|=2$. Since $|R_{12} \cap \text{im}(\tau_0)|=1$ and $|R_{34} \cap \text{im}(\tau_0)|=1$, we must have that $|D_{1} \cap \text{im}(\tau_0)|=1$ and $|D_{2} \cap \text{im}(\tau_0)|=1$. It follows as in Case (i) that since $K_1$ and $K_2$ are linearly prime, $B_{12} \cap \text{im}(\tau_0)$ is an unknotted arc. By Lemma \ref{corner_lemma}, there is a  vertically proper embedded disc $R_{34}'$ isotopic to $R_{34}$ and satisfying the property that $R_{12} \cap R_{34}'$ has fewer connected components than $R_{12} \cap R_{34}$. This contradicts the choice of $R_{12}$ and $R_{34}$.
\newline
\newline
\underline{Sub-case (iii.c):} Suppose that $|E_{12} \cap \text{im}(\tau_0)|=1$ and either $|D_1 \cap \text{im}(\tau_0)|=1$ or $|D_1 \cap \text{im}(\tau_0)|=1$ (but not both). By symmetry, it suffices to only consider the case that $|D_1 \cap \text{im}(\tau_0)|=1$ and $|D_2 \cap \text{im}(\tau_0)|=0$. The intersection of $\tau_0$ with $E_{12}$ must be either $\tau_0(0)$ or $\tau_0(1)$. Since $D_1$ contains $\tau_1'(1)=\tau_2'(0)$, it follows that either $\text{im}(\tau_1') \subset B_{12}$ or $\text{im}(\tau_2') \subset B_{12}$. Using the bead-on-a-string isotopy (Lemma \ref{beadonastring}) applied to $B_{12}$, it follows that either $(\Sigma_1',\tau_1')$ or $(\Sigma_2',\tau_2')$ stabilizes to a long classical knot. However, $K_1$ and $K_2$ are non-classical. This is a contradiction.
\newline
\newline
Thus, $R_{12} \cap R_{34}$ can contain no connected components which are cornered intervals.
\newline
\newline
\underline{Case (iv):} By Cases (i), (ii), and (iii), we may assume that $R_{12} \cap R_{34}$ contains no connected components which are embedded circles, partisan intervals or cornered intervals.  Let $e$ be a horizontal interval in $R_{12} \cap R_{34}$. Then $R_{12} \cap R_{34}$ contains no vertical intervals. Hence $R_{12} \cap R_{34}$ is a horizontal bunch. Since $R_{12}$ is disconnecting, there is exactly one connected component $C$ of $(\partial \Sigma_0) \times I$ such that $R_{12} \cap C \ne \emptyset$. We apply Lemma \ref{horiz_bunches} to $R_1=R_{12}$ and $R_2=R_{34}$. Let $j,k$ be as given in Lemma \ref{horiz_bunches} (1). There are two sub-cases to consider:
\begin{enumerate}[(a)]
\item $1<j,k<n+1$, or
\item either $j=k=1$ or $j=k=n+1$.
\end{enumerate}
Each of the sub-cases has sub-sub-cases.  We argue that each of the sub-sub-cases leads to a contradiction.
\newline
\newline
\underline{Sub-case (iv.a):} In this case we must have that the number of horizontal intervals $n$ is at least $2$. Let $D_1$ and $D_2$ be as given in Lemma \ref{horiz_bunches} (3). Let $B_{12}$ be the 3-ball from Lemma \ref{horiz_bunches} (3.a).  There are five non-trivial sub-sub-cases:
\begin{enumerate}[(1)]
\item $B_{12} \cap \text{im}(\tau_0)=\emptyset$,
\item $|(r_1^j \cup r_2^k) \cap \text{im}(\tau_0)|=2$,
\item $|(r_1^j \cup r_2^k) \cap \text{im}(\tau_0)|=1$ and $|(D_1 \cup D_2) \cap \text{im}(\tau_0)|=1$.
\item $|(D_1 \cup D_2) \cap \text{im}(\tau_0)|=2$ and $(r_1^j \cup r_2^k) \cap \text{im}(\tau_0)=\emptyset$, or
\item $|(D_1 \cup D_2 \cup r_1^j \cup r_2^k) \cap  \text{im}(\tau_0)|=4$
\end{enumerate}
We will show that each of the sub-sub-cases leads to a contradiction.
\newline
\newline
\underline{Sub-sub-case (iv.a.1):} Suppose that $B_{12} \cap \text{im}(\tau_0)=\emptyset$. Let $R_{1}'$ be the vertically proper embedded disc given in Lemma \ref{horiz_bunches} (3.b). Since $R_1'$ and $R_1=R_{12}$ are ambient isotopic and $R_1' \cap R_2$, has fewer connected components than $R_1 \cap R_2$, we have a contradiction of the choice of $R_{12}$ and $R_{34}$.   
\newline
\newline
\underline{Sub-sub-case (iv.a.2):} Suppose that $|(r_1^j \cup r_2^k) \cap \text{im}(\tau_0)|=2$. Then we must have that $|r_1^j \cap \text{im}(\tau_0)|=1$ and $|r_2^k \cap \text{im}(\tau_0)|=1$. It follows as in Case (i) that since $K_1$ and $K_2$ are linearly prime, then $B_{12} \cap \text{im}(\tau_0)$ is an unknotted arc. If we take $R_1'$ as in Lemma \ref{horiz_bunches} (3.b), we have a vertically proper embedded surface which is ambient isotopic to $R_1=R_{12}$ and satisfies the property that $R_1' \cap R_{34}$ has fewer connected components than $R_{12} \cap R_{34}$. This contradicts the choice of $R_{12}$ and $R_{34}$. 
\newline
\newline
\underline{Sub-sub-case (iv.a.3):} Suppose that $|(r_1^j \cup r_2^k) \cap \text{im}(\tau_0)|=1$ and $|(D_1 \cup D_2) \cap \text{im}(\tau_0)|=1$. By symmetry, it is sufficient to consider only the case that $|(r_1^j \cap \text{im}(\tau_0)|=1$ and $|D_1 \cap \text{im}(\tau_0)|=1$. Now, either $\tau_0(0)$ or $\tau_0(1)$ is in $D_1$. Also, since $\tau_1'(1)=\tau_2'(0)$, we must have that either $\text{im}(\tau_1') \subset B_{12}$ or $\text{im}(\tau_2') \subset B_{12}$. Applying the bead-on-a-string isotopy, we conclude that either $(\Sigma_1',\tau_1')$ or $(\Sigma_2',\tau_2')$ stabilizes to a long classical knot. This contradicts the hypotheses on $K_1$ and $K_2$. 
\newline
\newline
\underline{Sub-sub-case (iv.a.4):} Suppose that $|(D_1 \cup D_2) \cap \text{im}(\tau_0)|=2$ and $(r_1^j \cup r_2^k) \cap \text{im}(\tau_0)=\emptyset$. Since $R_{12}$ is disconnecting, we cannot have that both $\tau_0(0)$ and $\tau_0(1)$ are in the same disc $D_1$ or $D_2$. Say that $\tau_0(0)\in D_1$ and $\tau_0(1) \in D_2$. However, this means that $\tau_0(0)$ and $\tau_0(1)$ are in the same connected component of $\Sigma_0 \times I$ after cutting along $R_{12}$. This contradicts the fact that $R_{12}$ is a decomposition surface.
\newline
\newline
\underline{Sub-sub-case (iv.a.5):} Suppose that $|(D_1 \cup D_2 \cup r_1^j \cup r_2^k) \cap  \text{im}(\tau_0)|=4$. Then each of $D_1$, $D_2$, $r_1^j$, and $r_2^k$ has exactly one intersection with $\tau_0$. Since $R_{12}$ is disconnecting, cutting $\Sigma_0 \times I$ along $R_{12}$ gives two components $A_1$ and $A_2$.  One of the components contains $\tau_0(1)$ and one contains $\tau_0(0)$.  However, this contradicts the fact that $B_{12}$ is contained in exactly one of the components $A_1$ and $A_2$.
\newline
\newline
\underline{Sub-case (iv.b):} Now suppose that $j=k=1$ or $j=k=n+1$. Let $D_1$, $D_2$, and $D_3$ be as given in Lemma \ref{horiz_bunches} (2). Let $B_{12}$ be the 3-ball from Lemma \ref{horiz_bunches} (2.b). As in Sub-case (iv.a). There are five non-trivial sub-sub-cases to consider.
\begin{enumerate}[(1)]
\item $B_{12} \cap \text{im}(\tau_0)=\emptyset$,
\item $|(r_1^j \cup r_2^k) \cap \text{im}(\tau_0)|=2$,
\item $|(r_1^j \cup r_2^k) \cap \text{im}(\tau_0)|=1$ and $|(D_1 \cup D_2) \cap \text{im}(\tau_0)|=1$.
\item $|(D_1 \cup D_2) \cap \text{im}(\tau_0)|=2$ and $(r_1^j \cup r_2^k) \cap \text{im}(\tau_0)=\emptyset$, or
\item $|(D_1 \cup D_2 \cup r_1^j \cup r_2^k) \cap  \text{im}(\tau_0)|=4$
\end{enumerate}
These sub-sub-cases follow in a similar manner to sub-sub-cases (iv.a.1)-(iv.a.5).
\newline
\newline
Thus, $R_{12} \cap R_{34}$ can contain no connected components which are horizontal intervals.
\newline
\newline
Thus, all connected components of $R_{12} \cap R_{34}$ are vertical intervals.  By the remarks at the beginning of the proof, this completes our task.
\end{proof}

\bibliographystyle{plain}
\bibliography{bib_comm}

\end{document}